\documentclass[12pt]{article}
\usepackage{stmaryrd}
\usepackage{amssymb}
\usepackage{amsmath}
\usepackage{amsfonts}
\usepackage{amsthm}
\usepackage[colorlinks=false,pdfpagemode=FullScreen]{hyperref}
\usepackage{graphics}
\usepackage[dvips]{graphicx}
\usepackage{ifthen}
\usepackage[latin1]{inputenc}
\usepackage{epsfig}
\usepackage{psfig}
\usepackage{psfrag}

\def\th{\theta}
\def\a{\alpha}
\def\r{\rho}
\def\t{\tau}
\def\b{\beta}
\def\d{\partial}
\def\de{\delta}
\def\g{\gamma}
\def\l{\lambda}

\def\o{\omega}

\def\C{\mathbb{C}}
\def\O{\Omega}

\def\D{{\cal D}}

\def\phi{\varphi}
\def\e{\varepsilon}

\def\R{\mathbb{R}}

\def\H{\mathbb{H}}
\def\Z{\mathbb{Z}}

\def\T{\mathbb{T}}

\def\M{\mathbb{M}}
\def\A{\mathcal{A}}

\def\det{\mbox{{\rm det}}\;}

\newtheorem{defi}{Definition}[section]
\newtheorem{theo}[defi]{Theorem}
\newtheorem{prop}[defi]{Proposition}
\newtheorem{lem}[defi]{Lemma}
\newtheorem{corol}[defi]{Corollary}
\newtheorem{remark}[defi]{Remark}
\newtheorem{ass}[defi]{Assumption}

\numberwithin{equation}{section}

\begin{document}

\renewcommand{\refname}{References}

\date{}
\title{High-frequency limit of the Maxwell-Landau-Lifshitz
system in the diffractive optics regime}
\author{LU Yong\footnote{Universit\'e Paris-Diderot (Paris 7), Institut de Math\'ematiques de Jussieu, UMR CNRS 7586; e-mail: luyong@math.jussieu.fr}}

\maketitle {\footnotesize {\bf Abstract} - We study semilinear
Maxwell-Landau-Lifshitz systems in one space dimension. For highly
oscillatory and prepared initial data, we construct WKB approximate
solutions over long times $O(1/\e)$. The leading terms of the WKB
solutions solve cubic Schr\"odinger equations. We show that the
nonlinear normal form method of Joly, M\'etivier and Rauch
\cite{JMR_TMB} applies to this context. This implies that the
Schr\"odinger approximation stays close to the exact solution of
Maxwell-Landau-Lifshitz over its existence time. In the context of
Maxwell-Landau-Lifshitz, this extends the analysis of Colin and
Lannes \cite{CD_DS} from times $O(|\ln \e|)$ up to $O(1/\e)$.

\medskip

\medskip

\medskip

\tableofcontents
\begin{section}{Introduction}

The state of a ferromagnet is described by the magnetization vector
$M$. The evolution of $M$ with time $t$ without damping is governed
by the Landau-Lifshitz equation
\begin{equation}\label{LL}
 \d_t M=-M\times H,
\end{equation}
where $H$ is the magnetic field. The electromagnetic field variables
obey Maxwell equations
\begin{equation}\label{M}
\left\{ \begin{aligned}
         \d_t E-\nabla \times H &= 0, \\
         \d_t H+\nabla \times E &= -\d_t M.\\
                          \end{aligned} \right.
\end{equation}
The Maxwell-Landau-Lifshitz system that we study is
\eqref{LL}-\eqref{M}. The physical constants have been set equal to
one. The spatial domain is $\R^3$. Observe that the
Maxwell-Landau-Lifshitz system admits a family of constant
solutions:
\begin{equation}
(E,H,M)_\a =(0,\a M_0, M_0),\nonumber
\end{equation}
for any $\a>0$ and any $M_0\in\R^3$. We are interested in the
solutions that can be written as small, slowly variable
perturbations of such constant solutions. The perturbations are
measured in terms of an arbitrarily small positive parameter $\e$.
We postulate the following form for the perturbations:
\begin{equation}
 E(t,x) = \e \tilde{E}(\e t,\e x),~
         H(t,x) = \a M_0+\e \tilde{H}(\e t,\e x),~
         M(t,x) = M_0+\e \tilde{M}(\e t,\e x).
                          \nonumber
                          \end{equation}

If $(E,H,M)$ is a solution of \eqref{LL}-\eqref{M}, then
$(\tilde{E}(t,x),\tilde{H}(t,x),\tilde{M}(t,x))$ satisfies the
following system
\begin{equation}\label{1st eq}
\left\{ \begin{aligned}
        & \d_t \tilde{E}-\nabla \times \tilde{H} = 0, \\
         &\d_t \tilde{H}+\nabla \times \tilde{E} = -\d_t \tilde{M},\\
         &\d_t \tilde{M}=-\e^{-1}M_0 \times \tilde{H}+\e^{-1} \a M_0 \times \tilde{M}-\tilde{M}\times
         \tilde{H}.\\
                          \end{aligned} \right.
                          \end{equation}

Introduce the perturbation unknown
$u=(\tilde{E},\tilde{H},\a^{\frac{1}{2}}\tilde{M}):[0,\infty)\times
\R^3 \rightarrow \R^9=\R^3\times\R^3\times\R^3$. Then the system
\eqref{1st eq} can be written as a symmetric quadratic hyperbolic
system in $u$:
\begin{equation}\label{sys 3d}
\d_t u+A(\d_x)u+\frac{1}{\e} L_0 u=B(u,u),
\end{equation}

where $A$ and $L_0$ are defined by
\begin{equation}\label{def A}
A(\d_x)=
\begin{pmatrix}
0 & -\d_x \times  & 0\\
\d_x \times & 0 & 0\\
0 & 0 & 0
\end{pmatrix},~~
L_0=
\begin{pmatrix}
0 & 0  & 0\\
0& -M_0 \times & \a^{\frac{1}{2}}M_0 \times\\
0&  \a^{\frac{1}{2}}M_0 \times & -\a M_0 \times
\end{pmatrix}.
\end{equation}
The bilinear form $B(\cdot,\cdot)$ is defined for any $
u=(u^1,u^2,u^3), v=(v^1,v^2,v^3) \in\R^3\times \R^3\times \R^3$ by
\begin{equation}\label{def B}
B(u,v)= \frac{1}{2}\begin{pmatrix}
0 \\
\a^{-\frac{1}{2}}(u^3\times v^2+v^3\times u^2)\\
-(u^3\times v^2+v^3\times u^2)
\end{pmatrix}.
\end{equation}
It is easy to check that $A(\d_x)=\sum_{j=1}^3A_j\d_j$ with $A_j$
symmetric and that $L_0$ is skew-symmetric.
We set $\a=1$ and $M_0=(1,0,0)$.\\

 In \cite{HL}, Leblond did formal computations for Maxwell-Landau-Lifshitz system,
and derived a Davey-Stewartson system in the high-frequency limit,
for diffractive times. Our \emph{goal} is to justify the
Davey-Stewartson approximation over long time interval of order
$O(1/\e)$. For technical reasons that are made clear in Section
1.1.3 and have to do with the issue of localizing the resonances, so
we restrict our problem to one spatial dimension, where the
Davey-Stewartson system degenerates into cubic Schr\"odinger
equation. In one spatial dimension, the equation \eqref{sys 3d}
becomes
\begin{equation}\label{sys org}
\d_t v+A(e_1)\d_y v+\frac{L_0 v}{\e}=B(v,v),
\end{equation}
where the space variable is $y\in \R^1$, and the vector $e_1\in\R^3$
is
\begin{equation}\label{def e1}
e_1=(1,0,0)^t.
\end{equation}

We consider highly oscillatory initial data of the form
\begin{equation}\label{ini org}
 v(0,y)=a(y) e^{i k y/\e}+\overline{a(y)} e^{-i k y/\e}+\e a_1(y,k y/\e)+\e^2 a_2(y,k
 y/\e),
 \end{equation}
where $a_1(y,\th)$ and $a_2(y,\th)$ are real-valued and
$2\pi$-periodic in $\th$.

The system \eqref{sys org} is semilinear symmetric hyperbolic. For
fixed $\e>0$, the local existence, uniqueness and regularity of
local-in-time solutions to (\ref{sys org})-(\ref{ini org}) in smooth
Sobolev spaces are classical.

\bigskip

 Our goal here is to describe the
solutions in the high-frequency limit $\e\rightarrow 0$ over long
times $O(1/\e)$. There are three main issues.

\smallskip

(1).~First, because of the fast oscillations, the Sobolev norm of
the initial datum is unbounded as $\e\rightarrow 0$: for any $s>0$,
$$|v(0)|_{H^{s}}=|a(y)e^{iky/\e}+... |_{H^{s}}=O(\e^{-s})\rightarrow\infty.$$
By the classical theory of Cauchy problem of symmetric hyperbolic
system, the existence time $t^*(\e)$ shrinks to $0$ as
$\e\rightarrow 0$.

\smallskip

We can go around this issue and show that the natural existence time
is $O(1)$ by introducing \emph{profiles} as in \cite{JMR_Ind} and
\cite{JMR_TMB}. This means we consider a representation of the
solutions by maps depending periodically on the fast variable
$\th\in \T=\R/2\pi \Z$:
\begin{equation}\label{profile}
v(t,y)=V(t,y,\th)|_{\th=\frac{(-\o t+k y)}{\e}},
\end{equation}
where $\o\in \R$ is chosen such that the couple $(\o,k)$ satisfies
dispersion relation $\det(-i \o+A(e_1)i k+ L_0)=0$. Then to solve
the Cauchy problem ($\ref{sys org}$)-($\ref{ini org}$), it is
sufficient to solve the following Cauchy problem in $V$:

\begin{equation}\label{sys prof}
\left\{ \begin{aligned}
                 &\d_t V+A(e_1)\d_yV+\frac{1}{\e}\{-\o\d_{\th}+A(e_1)k\d_{\th}+ L_0\}
V=B(V,V),\\
                 &V(0,y,\th)=a(y)e^{i\th}+\overline{a(y)}e^{-i\th}+\e a_1(y,\th)
                 +\e^2 a_2(y,\th).\\
                          \end{aligned} \right.
                          \end{equation}

This is a symmetric hyperbolic system, for which the classical
theory gives local-in-time well-posedness, uniformly in $\e$.
Indeed, the data are now uniformly bounded with respect to $\e$ in
Sobolev spaces with non-negative indices, the $L^2$ estimate is
uniform in $\e$ in spite of the large $1/\e$ prefactor because the
operator $-\o\d_{\th}+A(e_1)k\d_{\th}+ L_0$ is skew-adjoint, and the
commutator estimates are trivial because the operator
$-\o\d_{\th}+A(e_1)k\d_{\th}+ L_0$ has constant coefficients.

\medskip

(2).~Second, the initial datum \eqref{ini org} has a \emph{large
amplitude} $O(1)$. In this regime, the maximal existence time
\emph{a priori} is $O(1)$, not $O(1/\e)$. For \eqref{sys prof}, the
classical $H^s(\R \times \T)$ energy estimate for semilinear
symmetric hyperbolic operators is
\begin{equation}
V(t)|_{H^s}\leq |V(0)|_{H^s}+C \int_0^t |V(t')|_{H^s} dt',\nonumber
\end{equation}
where the constant $C$ depends on $|V|_{L^\infty}$, implying, by
Gronwall's lemma, the bound
$$|V(t)|_{H^s}\leq  |V(0)|_{H^s}e^{Ct}.$$

Over long times $O(1/\e)$, the upper bound diverges to infinity as
$\e\rightarrow0$. We will overcome this problem by normal form
reduction and a rescalling method, following  Joly, M\'etivier and
Rauch in \cite{JMR_TMB} and Texier in \cite{em3}. This is in
contrast to the situation in which the initial data have a small
amplitude $O(\e)$,  where the regime is said to be \emph{weakly
nonlinear} and the natural existence time is $O(1/\e)$.

\medskip

(3).~Third, we are confronted with the issue of deriving uniform
bounds for the corrector terms of the ansatz. If we denote $V=V^0+\e
V^1 +...$ the WKB expansion, we want indeed to guarantee that $\e
V^1$ is $o(1)$ in the time interval under consideration, otherwise
the significance of the WKB expansion as an asymptotic expansion
breaks down. However there typically arises an equation of the form
${\displaystyle(\d_t+\r\d_y)a=b}$, where $a$ is a component of $V^1$
and $b$ depends on the leading term $V^0$. For the trivial datum
$a(0,y)=0$, the solution is $\displaystyle{a(t,y)=\int_0^tb(t-s,y-\r
s)ds},$
 giving typically the estimate $
\|a(t)\|_{H^s} \leq t \|b\|_{L^\infty (H^s)} $  and nothing better.
Over long time $O(1/\e)$, the upper bound is of order $O(1/\e)$
implying that $\e V^1$ could well be $O(1)$. This phenomenon is
called secular growth and was extensively studied by Lannes in
\cite{LSG}.
\bigskip

For Maxwell-Landau-Lifshitz, we overcome the issues (2) and (3), and
show existence up to times of order $O(1/\e)$. We also show that the
solution can be well approximated by WKB solutions, the leading
terms of which satisfy cubic Schr\"odinger equations.
\medskip

In the rest of this introduction, we describe in greater detail the
issues associated with (2)-(3) above in upcoming Section 1.1, before
turning to the statement of our results (in Section 1.2).

\subsection{Context}

\subsubsection{Secular growth}

In \cite{CD_DS}, Colin and Lannes study systems of the form
\eqref{sys 3d} with highly oscillatory initial data
\begin{equation}\label{ini 3d}
 u(0,x) =a(x) e^{i \vec{k} \cdot x/\e}+\overline{a(x)} e^{-i \vec{k} \cdot
 x/\e}+\e a_1(x,\frac{\vec{k} \cdot x}{\e})+\e^2 a_2(x,\frac{ \vec{k} \cdot x}{\e})+...,
 \end{equation}
corresponding to \eqref{ini org} in several space dimensions.

They first show, under a weak transparency assumption (see
Assumption 2.1 in \cite{JMR_TMB}) ensuring that a WKB cascade exists
that an a three-scale approximate profile of the form
\begin{equation}
V^a(t,x, \th)=\sum_{j=0}^{2} \e^j V_j(\t,t,x, \th)|_{\t=\e t},
~~V_j(\t,t,x, \th)=\sum_{p}\e^{ip\th}V_{jp}(\t,t,x)\nonumber
\end{equation}
can be constructed by three-scale WKB expansion over long times of
order $O(1/\e)$. The leading terms of the approximate solution solve
Davey-Stewartson systems. Under the additional assumption that the
Davey-Stewartson systems are well-posed, this proves existence of
WKB solutions over time $O(1/\e)$. As mentioned in Section 1.1,
Colin and Lannes have to deal with the issue of secular growth.
There arises in their analysis an equation of the form
\begin{equation}\label{sec eq2}(\d_t+\o'(k)\cdot\d_x)\pi_1 V_{11}=2\pi_1 B(V_{01},(\pi_0
V_{10})^*),\end{equation}
 where $\pi_0$ is a spectral projector (see \eqref{pi0}) and $(\pi_0
V_{10})^*$ is from the following decomposition
$$
\pi_0V_{10}=(\pi_0 V_{10})^*+\langle \pi_0 V_{10}\rangle
$$
where the part $\langle \pi_0 V_{10}\rangle$ satisfies the transport
equations $(\d_t+\o'(k)\cdot\d_x)\langle \pi_0 V_{10}\rangle=0$ and
the residual $(\pi_0 V_{10})^*$ satisfies the following homogeneous
hyperbolic equation
\begin{equation}\label{v10 star} (\d_t+\pi_0A(\d_x)\pi_0)(\pi_0
V_{10})^*=0.
\end{equation}
 The couple ($V_{01}$, $\langle \pi_0 V_{10}\rangle$) is the solution of
a Davey-Stewartson system and is uniformly bounded in a Sobolev
space over time interval $[0,T]_\t\times \R_t$, where the vector
$\langle \pi_0 V_{10}\rangle$ is determined by the leading term
$V_{01}$.

Given any initial datum to $(\pi_0 V_{10})^*$, the equation
\eqref{v10 star} implies a global solution $(\pi_0 V_{10})^*$. Back
to \eqref{sec eq2}, one may only generally have an estimate $\pi_1
V_{11}=O(t)$, unbounded as $t \to \infty$. Assuming that the initial
datum $a\in \Sigma^s=\{u\in H^{s+n^+/2}, (1+|\cdot|)^s \hat{u}\in
L^\infty \}$ for some $s>3+n/2$, Colin and Lannes prove a better
secular growth estimate in $O(\ln t)$ for $\pi_1 V_{11}$ (this is
Proposition 2.3 in \cite{CD_DS}). Then, assuming that the initial
perturbation $(v-v^a)(0)$ is $O(\e^2)$, Colin and Lannes prove
stability of the WKB solution in times $O(|\ln \e|)$ (this is
Proposition 3.1 in \cite{CD_DS}).
\medskip

Here, we consider the problem in one space dimension, with initial
data in classical Sobolev spaces. We allow initial perturbations of
size $O(\e),$ and give two results (see Theorem \ref{theorem} below
for a precise statement), the proofs of which rely heavily on the
specific structure of the Maxwell-Landau-Lifshitz system.


\medskip
First, assuming only the standard polarization condition for the
leading amplitude $a$ in the initial datum \eqref{ini org}, that is
$a(y)\in \ker(-i \o+A(e_1)i k+ L_0)$, we give a stability result in
intermediate times. This is statement ii in Theorem 1.4 below. We
show that a WKB approximate solution can be constructed, of which
the leading term solves cubic Schr\"odinger equation. We choose
$(\pi_0 V_{10})^* (0)= 0$ to eliminate the secular growth of $\pi_1
V_{11}$. Indeed, with this null initial datum, the equation
\eqref{v10 star} implies that $(\pi_0 V_{10})^*$ is identically
zero. Then the source term in \eqref{sec eq2} is null, hence no
secular growth for $\pi_1 V_{11}$. The initial error is of order
$O(\e)$:
\begin{equation}\label{ini diff0}
|(v-v^a)(0)|_{L^\infty}=O(\e),
\end{equation}
where the functions $v$ and $v^a$ are respectively the exact
solution and the approximate solution. We then prove the error
estimate in time $O(\e^\a|\ln \e|/\e)$, for any $0 < \a <1$:
$$
|(v -v^a)(t)| = O(\e^\a |\ln \e|).
$$

\medskip

Second, for the above WKB solution, we give a stability result in
diffractive times $O(1/\e)$, corresponding to statement iii in
Theorem 1.4 below. We assume additionally that the correctors are
prepared as follows
\begin{equation}\label{ini diff}
|\Pi_s(v-v^a)(0)|_{L^\infty}=O(\e^2),
\end{equation}
where the eigen-projectors $\Pi_0$ and $\Pi_s$ are defined later in
\eqref{pi0s}. In the context of our WKB analysis in Section 2, the
condition \eqref{ini diff} is equivalent to the following condition
(see also \eqref{exp a1}):
\begin{equation}\label{ini-choix}
a_1(y,\th)=(1-\pi_0)V_{10}(0)+\langle \pi_0
V_{10}\rangle(0)+\tilde{a}(y,\th),
\end{equation}
where $a_1$ is the first order corrector of the initial datum
\eqref{ini org}, and $\tilde{a}$ is a vector function that satisfies
$\Pi_s \tilde{a}=0$. Under the preparation condition \eqref{ini
diff}, using the specific structure of Maxwell-Landau-Lifshitz (see
the upcoming Section 1.1.2), we then prove that, up to times
$O(1/\e)$, the approximate solution is stable in the following sense
$$
\sup_{t\in [0,T/\e],y\in \R}|\Pi_0(v-v^a)|=O(\e),~~\sup_{t\in
[0,T/\e],y\in \R}|\Pi_s(v-v^a)|=O(\e^2),
$$
meaning that the initial errors \eqref{ini diff0} and \eqref{ini
diff} are essentially preserved by the dynamics of the
Maxwell-Landau-Lifshitz system in the diffractive regime.

The condition \eqref{ini diff} is weaker than the corresponding
condition (24) in \cite{CD_DS}. Indeed, we only give constraints on
the component of the first corrector that belongs to the image of
$\Pi_s$. Unlike the condition (24) in \cite{CD_DS}, $V_{11}(0)$ does
not appear in \eqref{ini diff}, because we have the equality $\Pi_s
V_{11}=0$ (see \eqref{exp u11}).

\subsubsection{Maxwell-Bloch structure}

In \cite{C}, Colin (following Joly, M\'etivier and Rauch
\cite{JMR_TMB}) studies systems of the form
\begin{equation}\label{mb}
\left\{ \begin{aligned}
         (\d_t+B_1(\d_x)+\frac{E}{\e})u=f(u,m), \\
         (\d_t+B_2(\d_x)+\frac{F}{\e})m=g(u,u). \\
                          \end{aligned} \right.
\end{equation}

This operator has a block-diagonal structure which is very
particular. This is the one which is used in \cite{JMR_TMB}. Its
characteristic variety is the union of the characteristic varieties
of $\d_t+B_1(\d_x)+E$ and $\d_t+B_2(\d_x)+F.$ Assume the spectral
decomposition
$$
B_1(\xi)+E/i=\sum_j \l_j(\xi)\Pi^1_j(\xi),~B_2(\xi)+F/i=\sum_j
\mu_j(\xi)\Pi^2_j(\xi),
$$
where $\l_j$ and $\mu_j$ are eigenvalues, $\Pi^1_j$ and $\Pi^2_j$
are eigenprojectors.

 Colin makes a \emph{strong transparency hypothesis} as in
Joly, M\'etivier and Rauch \cite{JMR_TMB} (see Assumption 2.10 in
\cite{JMR_TMB}), saying that

\begin{ass}\label{st assu}
There exists a constant $C>0$ such that for all $\xi_1,\xi_2,\eta
\in \R^d$ which satisfy $\eta=\xi_1+\xi_2$, for any eigenvalues
$\mu_j(\eta),~\l_{j'}(\xi_1),~\l_{j''}(\xi_2)$ respectively of
$B_2(\eta)+F/i$, $B_1(\xi_1)+E/i$ and $B_1(\xi_2)+E/i$ and for all
$u,~v\in \C^n$, one has
\begin{equation}\label{sttrans}
\big|\Pi^2_j(\eta)g\big(\Pi_{j'}^1 (\xi_1)u,\Pi_{j''}^1
(\xi_2)v\big)\big|\leq
C|\l_{j'}(\xi_1)+\l_{j''}(\xi_2)-\mu_j(\eta)||u||v|.
\end{equation}

\end{ass}

We say a system has \emph{Maxwell-Bloch} \emph{structure}, if it is
of the form (\ref{mb}) and satisfies Assumption \ref{st assu}.

\medskip

For systems with Maxwell-Bloch structure, Colin shows that a WKB
expansion can be performed on (\ref{mb}), where the leading terms of
the three-scale WKB solution satisfy an elliptic Davey-Stewartson
system which is locally well-posed, implying that an approximate
solution $(u^a,m^a)$ can be constructed over times $O(1/\e)$ (see
Proposition 4.1 in \cite{C}). Assuming additionally that the initial
correctors are identically zero, that is, $(u,m)(0)=(u^a,m^a)(0)$
(which is strictly stronger than \eqref{ini diff}), Colin showed a
convergence result over times $O(1/\e)$. The Maxwell-Bloch structure
gives a control on the interactions of resonant waves (see Section
4.2 in \cite{C} and Section 3.2 in this paper). In this respect, it
is crucial in the derivation of error estimates in time $O(1/\e)$.

\bigskip

Here, we check that the Maxwell-Landau-Lifshitz system in one space
dimension \eqref{sys org} is of Maxwell-Bloch type.

\subsubsection{Resonances}
In order to check Assumption 1.1 on \eqref{sys 3d}, we need to
describe the resonances. For any $\xi=(\xi_1,\xi_2,\xi_3)$, we have
the spectral decomposition

$$
\sum_{m=1}^3 A_m \xi_m +L_0 /i=\sum_{j=1}^9 \l_j(\xi)\Pi_j(\xi).
$$
The resonance set is defined as the domain in frequency space in
which the upper bound in \eqref{sttrans} vanishes, that is the set
of frequencies $(\xi,\eta)$ such that $\Phi_{j_1,j_2,j_3}(\xi,\eta)
= 0,$ where the phases functions are defined by
$$
\Phi_{j_1,j_2,j_3}(\xi,\eta)=\l_{j_1}(\xi+\eta)-\l_{j_2}(\xi)-\l_{j_3}(\eta).
$$

For \eqref{sys 3d}, the eigen-polynomial is
$$
\det \big(-\l+\sum_{j=1}^3 A_j \xi_j +L_0
/i\big)=\l^3\big[\l^6-2(2+|\xi|^2)\l^4+(|\xi|^2(6+|\xi|^2)-2\xi_1^2)\l^2-|\xi|^2(2|\xi|^2-\xi_1^2)\big].
$$
It is of degree nine with three parameters $\xi_1$, $\xi_2$ and
$\xi_3$. The eigenvalues $\l_j$ do not symmetrically depend on
$\xi_1$, $\xi_2$ and $\xi_3$. This property makes the resonance
equation $\Phi_{j_1,j_2,j_3}(\xi,\eta) = 0$ too complicated to solve
analytically. For this technical reason, we consider the one space
dimensional system (\ref{sys org}), in which the dispersion relation
reduces to a simpler formula (see below \eqref{dis rel}).

\subsection{Description of the results}

If a $C^1$ profile $V(t,y,\th)$ solves the Cauchy problem
\begin{equation}\label{1d sys pro}
\left\{ \begin{aligned}
                 &\d_{t}V+\frac{1}{\e}A(e_1)\d_yV+\frac{1}{\e}\{-\o\d_{\th}+A(e_1)k \d_{\th}+ L_0\}
V=B(V,V),\\
                 &V(0,y,\th)=a(y)e^{i\th}+\overline{a(y)}e^{-i\th}+\e a_1(y,\th)+\e^2 a_2(y,\th)\\
                          \end{aligned} \right.
                          \end{equation}
on the time interval $[0,T/\e]$, then the function $v(t,y)$ defined
as \eqref{profile} solves the Cauchy problem (\ref{sys
org})-(\ref{ini org}) on the time interval $[0,T/\e]$.

\medskip

Throughout the paper, we always make the following assumptions:

We assume that the phase ($\o,k$) satisfies $\o \neq 0$ and the
dispersion relation $\det L(i(\o,k))=0$ with $L(i(\o,k)):=(-i
\o+A(e_1)i k+ L_0)=0$. This means that we choose a time frequency
that is compatible with the initial space frequency. We assume that
the leading term of the initial datum satisfies the polarization
condition: $a(y)\in \ker L(i(\o,k))$.  We also assume the regularity
conditions $a(y) \in H^s$, $a_1(y, \th)\in
H^1(\T_\th,H^{s-1}(\R^1_y))$ and $a_2(y,\th)\in
H^1(\T_{\th},H^{s-2}(\R^1_y)) $, where $s>2+1/2$.

\medskip

In Section 2, we construct an approximate profile for (\ref{1d sys
pro}) by WKB expansion. This construction relies on the structural
condition \eqref{weak str}, which implies the weak transparency
condition of Joly, M\'etivier and Rauch.

\begin{prop}\label{exis app1d}{(Existence of Approximate Profile)}\\
Under the above assumptions, there exists ${\displaystyle V^a\in
L^\infty ([0,T/\e]_t, H^1(\T_{\th},H^{s-2}(\R^1_y)))}$ for some
positive $T>0$ independent of $\e$, such that $V^a$ solves the
following Cauchy problem on time interval $[0,T/\e]_t$:
\begin{equation}\label{1d sys app}
\left\{ \begin{aligned}
                 &\d_{t}V^a+
A(e_1)\d_yV^a+\frac{1}{\e}\{-\o\d_{\th}+A(e_1)k\d_{\th}+ L_0\}
V^a=B(V^a,V^a)+\e^2 R\\
                 &V^a(0,y,\th)=V(0,y,\th)+\e b(y,\th)+\e^2 b_1(y,\th)\\
                          \end{aligned} \right.
                          \end{equation}
for some $R(t,y,\th)\in L^\infty ([0,T/\e]_t,
H^1(\T_{\th},H^{s-2}(\R^1_y))),~(b(y,\th), b_1(y,\th))\in
H^1(\T_{\th},H^{s-2}(\R^1_y))$.

Precisely, the function $ b$ is given in {\rm(\ref{exp b})} and has
the property stated in {\rm Lemma \ref{pt b}}; the function $R$ is
given by {\rm (\ref{exp r})} and has the property stated in {\rm
Lemma \ref{pt r}}. Moreover, the leading term of the WKB approximate
profile $V^a$ solves a cubic Schr\"odinger equation.

\end{prop}

We remark that it does not seem possible to construct a more precise
WKB solution. Indeed, for higher order terms, there arises an
equation of the form
\begin{equation}\label{ds v30}
 (-\r +\pi_0A(e_1)\pi_0)\d_y \langle V_{20}\rangle=4{\rm{Re}}\pi_0 B(V_{0,-1},V_{21})+2{\rm{Re}}\pi_0B(V_{1,-1},V_{11})-\d_\t \langle
 V_{10}\rangle,
\end{equation}
where the right hand side is not of the form $\d_y \mathcal{R}$,
then equation \eqref{ds v30} is not well posed in classical Sobolev
spaces. This can be contrasted with corresponding equations for
lower order terms (see \eqref{ds ch2e3}), where the source terms are
indeed of the form $\d_y \mathcal{R}$.

\medskip

For $V^a$ given by Proposition \ref{exis app1d}, defining
${\displaystyle v^a:=V^a\big(t,x,\frac{ky-\o t}{\e}\big)}$, then
$v^a \in L^\infty ([0,T/\e]_t \times\R^1_y)$ and satisfies the
following Cauchy problem on time interval $[0,T/\e]$:

\begin{equation}\label{equ app}
\left\{ \begin{aligned}
                 &(\d_t
+A(e_1)\d_y)v^a+\frac{1}{\e}L_0
v^a=B(v^a,v^a)+\e^2 R,\\
                 &v^a(0,y)=v(0,y)+\e b(y,\frac{ky}{\e})+\e^2 b_1(y,\frac{ky}{\e}).\\
                          \end{aligned} \right.
                          \end{equation}

We then justify the WKB expansion by giving an upper bound for
$|v-v^a|$ that is valid over the existence time for $v^a$:

\begin{theo}\label{theorem}
 Under the assumptions stated just above Proposition {\rm \ref{exis app1d}},
 there exists $T^*>0$ and $0<\e_0<1$, such that for all $T<T^*$, for
some constant $C=C(s,T)>0$ depending on $s$ through
$\|a\|_{H^s}+\|a_1\|_{H^2(\T_\th,H^{s-1})}+\|a_2\|_{H^2(\T_\th,H^{s-2})}$,
we have that for all $0<\e<\e_0$:

{\rm \textbf{i.}} Over the time interval $[0,T/\e]$, the Cauchy
problem {\rm\eqref{sys org}-\eqref{ini org}} admits a unique
solution $v$ of the form ${\displaystyle v(t,x)=V(t,x,\frac{ky-\o
t}{\e})}$, with $V(t,x,\th)\in L^\infty ([0,T/\e]_{t},
H^1(\T_{\th},H^{s-2}(\R^1_y)))$.

 {\rm \textbf{ii.}}We have the following error estimates
\begin{equation}\label{est1}
\|\Pi_0(v-v^a)\|_{L^\infty([0,\frac{T}{\e}]\times \R^1)}\leq C
(\e+T), ~~\|\Pi_s(v-v^a)\|_{L^\infty([0,\frac{T}{\e}]\times
\R^1)}\leq C \e,
\end{equation}

{\rm \textbf{iii.}} Assuming additionally that $a_1$ satisfies
\eqref{ini-choix}, we have the better error estimates
 \begin{equation}\label{est2}
\|\Pi_0(v-v^a)\|_{L^\infty([0,\frac{T}{\e}]\times \R^1)}\leq C \e,
~~\|\Pi_s(v-v^a)\|_{L^\infty([0,\frac{T}{\e}]\times \R^1)}\leq C
\e^2.
\end{equation}

\end{theo}

\medskip

We remark that for times just short of the diffractive time scale,
that is $T = \e^\a |\ln \e|,0<\a<1$, the error estimate \eqref{est1}
gives an upper bound that is $o(1)$ in the limit $\e \to 0:$
\begin{equation}\label{est11}
\|(v-v^a)\|_{L^\infty([0,\e^{-1+\a} |\ln \e|]\times \R^1)}=O(
\e^{\a}|\ln \e|).
\end{equation}

\medskip

The proof of Theorem \eqref{theorem} is decomposed into two steps.
First in Section 3, we show that the system \eqref{1d sys pro} has
the Maxwell-Bloch structure, then by a nonlinear change of variable
introduced by Joly, M\'etivier and Rauch in \cite{JMR_TMB}, we prove
that the Cauchy problem \eqref{1d sys pro} admits a unique solution
over times of order $O(1/\e)$. Then in Section 4, we prove the error
estimates \eqref{est1} and \eqref{est2}.

\bigskip

We conclude this introduction by setting up some notations:

 For all $(\l,\xi)\in\R^{1+1}$, we introduce the matrices
$L(i(\l,\xi))=i\l+A(e_1)i\xi+L_0$ and the characteristic variety
${\rm Char} L(i(\l,\xi))=\{(\l,\xi)|\det L(i(\l,\xi))=0\}.$

\medskip

 For $\b=(\o,k)$, denote $L(\b\d_\th)=-\o \d_\th +A(e_1)k\d_\th
+L_0$ and $L_p=L(ip\b)=-ip\o+ipkA(e_1)+L_0$. Denote by $\pi_p$ the
orthogonal projector onto $\ker L_p$ and $L_p^{-1}$ the inverse of
$L_p$, or, if $L_p$ is singular, its partial inverse which is
defined as
$$\pi_pL_p^{-1}=L_p^{-1}\pi_p=0,
~L_pL_p^{-1}=L_p^{-1}L_p={\rm Id}-\pi_p.$$

\medskip

 For all $\xi\in\R^1$, we have the spectral
decomposition
\begin{equation}\label{sp decom}
A(\xi e_1)+L_0/i=\sum_{j=1}^9 \l_j(\xi)\Pi_j(\xi).
\end{equation}

We denote the total eigenprojectors
\begin{equation}\label{sp decom2}
\Pi_0(\xi)=\sum_{j=1}^6 \Pi_j(\xi), ~\Pi_s(\xi)=\sum_{j'=7}^9
\Pi_{j'}(\xi).
\end{equation}
By direct calculation, for all $\xi\in \R$, we have
$\l_7(\xi)=\l_8(\xi)=\l_9(\xi)=0$ and
\begin{equation}\label{pi0s}
\Pi_0(\xi)={\rm diag}\{0,1,1,0,1,1,0,1,1\},~~\Pi_s(\xi)={\rm
diag}\{1,0,0,1,0,0,1,0,0\}.
\end{equation}
The projector $\Pi_0(\xi)$ and $\Pi_s(\xi)$ are actually constant
matrices, independent of the variable $\xi$.

\medskip

In the following, most functions and operators depend on $\e$. In
this context, we typically use \emph{bounded} to mean \emph{bounded
uniformly in} $\e\in (0,1)$. Later on, we restrict the range of $\e$
to $0<\e<\e_0$, for some $\e_0$ small enough.

\end{section}

\begin{section}{WKB expansion and approximate profile}

The aim of this section is to construct an approximate profile of
(\ref{1d sys pro}) through WKB expansion. In the nonlinear regime of
our interest, the limit equations are cubic Schr\"odinger equations.
\begin{subsection}{The WKB Expansion}
We look for an approximate profile $V^a$ of the form
\begin{equation}\label{wkb sol 1d}
V^a(t,y, \th)=\sum_{j=0}^{2} \e^j V_j(\e t,t,y, \th), ~~V_j(\t,t,y,
\th)=\sum_{p\in \Z}\e^{ip\th}V_{jp}(\t,t,y).
\end{equation}

We plug (\ref{wkb sol 1d}) into (\ref{1d sys pro}) and cancel all
the terms of orders $O(\e^j)$, $j=-2,-1,0$.

\medskip

\textbf{Equations for the terms in  $O(\e^{-2})$}:
~$L(\b\d_{\th})V_0= (-\o \d_\th +A(e_1)k\d_\th +L_0)V_0=0 $. By
\eqref{wkb sol 1d}, it amounts to solving the following equations
for all $p\in \Z$:
\begin{equation}\label{eq u0p}
L_p V_{0p}=(- ip \o +A(e_1) ip k +L_0)V_{0p} =0.
\end{equation}
We study (\ref{eq u0p}) for each $p\in \Z$.

\medskip

\fbox{\textbf{p=0}} If $p=0$, the equation (\ref{eq u0p}) becomes
$L_0V_{00}=0$. Since there is no mean mode in the leading term of
the initial datum (\ref{1d sys pro}), we choose naturally
$V_{00}=0$. Actually, assuming the leading mean mode $V_{00}$ to be
null is crucial to obtain a Davey-Stewartson approximation, as in
\cite{C} and \cite{CD_DS} (here the Davey-Stewartson approximation
degenerates into a nonlinear Schr\"odinger approximation).

\medskip

\fbox{\textbf{p=1}} If $p=1$, the equation (\ref{eq u0p}) becomes
$L_1V_{01}=0$. By direct calculation, the dispersion relation $\det
L_1=0$ is
\begin{equation}\label{dis rel1}
\o^2(1-\frac{k^2}{\o^2})^2=(2-\frac{k^2}{\o^2})^2.
\end{equation}
We choose $\o\neq0$ and $k\neq 0$ that satisfy the dispersion
relation \eqref{dis rel1}. We denote
\begin{equation}\label{de l}
\de=-\frac{\o(1-\frac{k^2}{\o^2})}{2-\frac{k^2}{\o^2}},
~~~\g=1-\frac{k^2}{\o^2},
\end{equation}
then the dispersion relation (\ref{dis rel1}) is equivalent to
\begin{equation}\label{dis rel2}
k^2=\frac{\o+2\de}{\o+\de}\cdot\o^2,~~\de=\pm1.\nonumber
\end{equation}
Through direct calculation, we find that $\ker L_1$ is a
one-dimensional vector space with generator $W_0$:
\begin{equation}\label{kerL1}
W_0=\begin{pmatrix}
-\frac{i\de k}{\o}\O_0 \\
\O_0\\
-\g\O_0
\end{pmatrix},~~
\O_0=\begin{pmatrix}
0 \\
i\de\\
1
\end{pmatrix}.
\end{equation}
Then the orthogonal projector $\pi_1$ onto $\ker L_1$ has the
expression:
\begin{equation}\label{exp pi1}
\pi_1 V= \frac{(V|W_0)}{|W_0|^2}W_0,~~{\rm for ~all~ }V \in \R^9.
\end{equation}
where $(\cdot|\cdot)$ denotes the inner product in the complex
vector space $\C^9$. Then the solution $V_{01}$ has the following
form: for some scalar function $g$ (which is determined later by
\eqref{g_1} and \eqref{schr1}):
\begin{equation}\label{exp v01}
V_{01}=W_0g
\end{equation}

\medskip

\fbox{$|\textbf{p}|\geq 2$} The equation \eqref{eq u0p} becomes
$L_pV_{0p}=0$. For nonzero $(\o,k)$, $\det L_p=0\Longleftrightarrow
p^2\o^2(1-\frac{k^2}{\o^2})^2=(2-\frac{k^2}{\o^2})^2$. By our choice
of $(\o,k)$ that satisfies (\ref{dis rel1}), the matrix $L_p$ is
invertible if $|p|\geq 2$. Then the solutions are trivial
$V_{0p}=0$, for all $|p|\geq 2$.

\medskip

\fbox{\textbf{$p\leq-1$}} By reality of the datum, the only natural
choice is $V_{j,p}=\overline{V}_{j,-p}.$

\bigskip

\textbf{Equations for the terms in  $O(\e^{-1})$}. For all $p$:
\begin{equation}\label{eq u1p}
(\d_t+A(e_1)\d_y)V_{0p}+L_p
V_{1p}=\sum_{p_1+p_2=p}B(V_{0p_1},V_{0p_2}).
\end{equation}
\fbox{\textbf{p=0}} If $p=0$, the equation (\ref{eq u1p}) reduce to
\begin{equation}\label{eq v10}
L_0V_{10}=2B(V_{01},V_{0,-1}).\nonumber\end{equation}

 We calculate
\begin{equation}
B(V_{01},V_{0,-1})=g\bar{g}B(W_0,\overline{W}_0)=\frac{g\bar{g}}{2}\begin{pmatrix}
0 \\
-\g\O_0\times \overline{\O}_0-\g\overline{\O}_0\times \O_0\\
\g\O_0\times \overline{\O}_0+\g\overline{\O}_0\times \O_0
\end{pmatrix}=0.\nonumber
\end{equation}
This means we have $ B(\pi_{1}\cdot,\pi_{-1}\cdot)=0. $ The
relationship $\pi_0B(\pi_{1}\cdot,\pi_{-1}\cdot)=0$, which is
necessary to construct the Davey-Stewartson approximate solution in
Colin and Lannes \cite{CD_DS}, is the \emph{weak transparency}
condition introduced by Joly, M\'etivier and Rauch \cite{JMR_TMB}.
Here, we even have the following stronger result
\begin{equation}\label{weak str}
B(\pi_{\pm 1}\cdot,\pi_{\pm1}\cdot)=0.
\end{equation}

Then the equation \eqref{eq v10} becomes $L_0V_{10}=0$. We denote
$\pi_0$ the projection onto $\ker L_0$. Then we have $
V_{10}=\pi_0V_{10}.$ Now we calculate the vector space $\ker L_0$
and the projector $\pi_0$. We let $V_{10}=(E_{10},H_{10},M_{10})^t$,
where $E_{10},~H_{10}$ and $M_{10}$ are vectors in $\R^3$. The
equations (\ref{def A}) and \eqref{eq v10} give
\begin{equation}\label{kerl0 0}
L_0V_{10}=0 \Longleftrightarrow M_0\times
(H_{10}-M_{10})=0\Longleftrightarrow (H_{10}-M_{10})\sslash
(1,0,0)^t.
\end{equation}
Then we obtain a basis of $\ker L_0$:
\begin{equation}\label{kerl0}
\ker L_0=\rm{span}\big\{\begin{pmatrix}e_1
\\0\\0\end{pmatrix},\begin{pmatrix} e_2\\0\\0\end{pmatrix},
\begin{pmatrix}e_3\\0\\0\end{pmatrix},
\begin{pmatrix}0\\e_1\\0\end{pmatrix},
\begin{pmatrix}0 \\e_2\\e_2\end{pmatrix},
\begin{pmatrix}0\\e_3\\e_3\end{pmatrix},
\begin{pmatrix}0\\0\\e_1\end{pmatrix}
\big\},\nonumber
\end{equation}
where $e_1=(1,0,0)^t,e_2=(0,1,0)^t,e_3=(0,0,1)^t$. The orthogonal
projector $\pi_0$ onto $\ker L_0$ is
\begin{equation}\label{pi0}
\pi_0=
\begin{pmatrix}
I & 0 & 0\\
0& J_1 & J_2\\
0 & J_2 & J_1
\end{pmatrix}
\end{equation}
 with the blocks
\begin{equation}\label{j1 j2}
I=\begin{pmatrix}
1 & 0 & 0\\
0& 1 & 0\\
0 & 0& 1
\end{pmatrix},~~J_1=
\begin{pmatrix}
1 & 0 & 0\\
0& \frac{1}{2} & 0\\
0 & 0& \frac{1}{2}
\end{pmatrix},~~J_2=\begin{pmatrix}
0 & 0 & 0\\
0& \frac{1}{2} & 0\\
0 & 0& \frac{1}{2}
\end{pmatrix}.
\end{equation}

\medskip

\fbox{\textbf{p=1}}~~~If $p=1$, the equation becomes
$(\d_t+A(e_1)\d_y)V_{01}+L_1V_{11}=\sum_{p_1+p_2=1}B(V_{0p_1},V_{0p_2})$.

\medskip

From the analysis for the terms in $O(\e^{-2})$, we have that
$V_{00}=0$ and $V_{0p}=0,$~for all~$ |p|\geq2$, which give us
$\sum_{p_1+p_2=1}B(V_{0p_1},V_{0p_2})=0.$ Then the equation for
$p=1$ is
\begin{equation}\label{eq u11}
(\d_t+A(e_1)\d_y)V_{01}+L_1V_{11}=0.
\end{equation}
Applying $\pi_1$ to \eqref{eq u11}, by identity $\pi_1 L_1 =0$, we
obtain
\begin{equation}\label{eq u01 org}
\pi_1(\d_t+A(e_1)\d_y)V_{01}=0,
\end{equation}
a transport equation for the leading profile. In the geometric
optics approximation, the transport occurs at group velocity (See
Proposition 2.6 of \cite{T0} and Proposition 2.2 in \cite{JMR_Ind}
). Here we compute the transport operator $\pi_1 A(e_1)\d_y\pi_1$
explicitly. By (\ref{exp pi1}), the equation \eqref{eq u01 org} is
equivalent to
\begin{equation}
((\d_t+A(\d_x))V_{01}|W_0)=0.\nonumber
\end{equation}
By (\ref{kerL1}) and (\ref{exp v01}), we have that
\begin{equation}
(\d_tV_{01}|W_0)=\d_t g~|W_0|^2=\d_t g~
2\big(1+\frac{k^2}{\o^2}+(1-\frac{k^2}{\o^2})^2\big)\nonumber
\end{equation}
and that
\begin{equation}
(A(e_1)\d_y V_{01}|W_0)=\frac{4k}{\o}\d_{y}g.\nonumber
\end{equation}
Then the equation \eqref{eq u01 org} becomes the following transport
equation in the scalar function $g$:
\begin{equation}\label{eq trans}
\d_tg+\r\d_yg=0
\end{equation}
with the group velocity
\begin{equation}\label{gv}
\r=\big(1+\frac{k^2}{\o^2}+(1-\frac{k^2}{\o^2})^2\big)^{-1}\frac{2k}{\o}.
\end{equation}
\\
On the other hand, by applying $L_1^{-1}$ to (\ref{eq u11}), one has
\begin{equation}\label{1-pi u11}
(1-\pi_1)V_{11}=-L_1^{-1}\big((\d_t+A(e_1)\d_y)V_{01}\big)=-L_1^{-1}\big(A(e_1)\d_yV_{01}\big).
\end{equation}

Explicitly, the solution of equation \eqref{eq u11} is
$V_{11}=(E_{11},H_{11},M_{11})^t$ with
\begin{equation}\label{exp u11}
\left\{ \begin{aligned}
                 &E_{11}=-\frac{i\de k}{\o}f \O_0+\frac{\de}{\o}(-1+\frac{k\r}{\o})\d_yg \O_0,\\
                 &H_{11}=f \O_0,\\
                 &M_{11}=-\g f \O_0+\frac{2ik}{\o^2}(-1+\frac{k\r}{\o})\d_yg\O_0,\\
                          \end{aligned} \right.
                          \end{equation}

where $f$ is an unknown function that is yet to be determined (see
\eqref{f h}), and is associated with the component $\pi_1V_{11}$ of
$V_{11}$.

\medskip

\fbox{\textbf{p=2}}~~~If $p=2$, the equation becomes
$(\d_t+A(e_1)\d_y)V_{02}+L_2V_{12}=\sum_{p_1+p_2=2}B(V_{0p_1},V_{0p_2})$.

\medskip

Since $V_{00}=0,~V_{0p}=0$, for all $|p|\geq2$ and (\ref{weak str}),
the equation for $p=2$ is actually $L_2V_{12}=B(V_{01},V_{01})=0,$
which gives trivial solution $V_{12}=0$ because the matrix $L_2$ is
invertible.

\medskip

\fbox{\textbf{p$\geq$3}}~~~The equations are
$(\d_t+A(e_1)\d_y)V_{0p}+L_pV_{1p}=\sum_{p_1+p_2=p}B(V_{0p_1},V_{0p_2})$,
and we directly obtain the solutions $V_{1p}=0$.\\

\fbox{\textbf{p$<$0}}~~~For negative $p$, we simply take $V_{jp}=\overline{V}_{j,-p}$.\\

\textbf{Equations for the terms in  $O(\e^{0})$}. For all $p$,
\begin{equation}\label{eq u2p}
\d_\t
V_{0p}+(\d_t+A(e_1)\d_y)V_{1p}+L_pV_{2p}=2\sum_{p_1+p_2=p}B(V_{0p_1},V_{1p_2}).
\end{equation}
\medskip
\fbox{\textbf{p=0}} For p=0, the equation (\ref{eq u2p}) becomes
\begin{equation}\label{eq u20}
\d_\t
V_{00}+(\d_t+A(e_1)\d_y)V_{10}+L_0V_{20}=2\sum_{p_1+p_2=0}B(V_{0p_1},V_{1p_2}).
\end{equation}
We calculate the right hand side
\begin{equation}
\sum_{p_1+p_2=0}B(V_{0p_1},V_{1p_2})=B(V_{01},V_{1,-1})+B(V_{0,-1},V_{11})=2{\rm
Re} B(V_{0,-1},V_{11}).\nonumber
\end{equation}
Together with the choice $V_{00}=0$, the equation (\ref{eq u20})
becomes
\begin{equation}\label{eq u201}
(\d_t+A(e_1)\d_y)V_{10}+L_0V_{20}=4{\rm Re}B(V_{0,-1},V_{11}).
\end{equation}
Applying $\pi_0$ to (\ref{eq u201}), we obtain that
\begin{equation}\label{ds 2eo}
 (\d_t+\pi_0A(e_1)\d_y\pi_0)\pi_0V_{10}=4\pi_0{\rm{Re}}B(V_{0,-1},V_{11}).
\end{equation}

By (\ref{def B}), (\ref{kerL1}), (\ref{exp v01}), (\ref{pi0}) and
(\ref{exp u11}), we calculate the source term of (\ref{ds 2eo}):
\begin{equation}\label{pi0 B}
4\pi_0{\rm Re}
B(V_{0,-1},V_{11})=\begin{pmatrix}0\\A_0\\-A_0\end{pmatrix}
\end{equation}
with
\begin{equation}\label{A0 0}
A_0=\begin{pmatrix}k_0\\0\\0\end{pmatrix},~k_0=\frac{4k\de}{\o^2}(1-\frac{k\r}{\o})\d_y|g|^2.
\end{equation}

By (\ref{def A}), (\ref{pi0}) and (\ref{j1 j2}), we have that
\begin{equation}\label{pi0 A pi0}
\pi_0A(e_1)\pi_0=\begin{pmatrix}
0 & (-e_1 \times)J_1  & (-e_1 \times)J_2\\
J_1(e_1 \times) & 0 & 0\\
J_2(e_1 \times) & 0 & 0
\end{pmatrix}.
\end{equation}
We denote $V_{10}=(E_{10},H_{10},M_{10})$, by (\ref{pi0 B}) and
(\ref{pi0 A pi0}), the equation (\ref{ds 2eo}) becomes

\begin{equation}\label{dsch2e1}
\left\{ \begin{aligned}
         &\d_t E_{10}-(e_1 \times)J_1 \d_y H_{10}  -(e_1\times)J_2 \d_y M_{10} =0,\\
                  &\d_t H_{10}+J_1(e_1 \times)\d_yE_{10}  =A_0,\\
                 &\d_t M_{10}+J_2(e_1 \times )\d_yE_{10}  =-A_0.
                          \end{aligned} \right.
                          \end{equation}

We now introduce a decomposition inspired from the averaging method
of Lannes \cite{Lav}: for any vector function $V(t,y)$, we use the
notation $\langle V \rangle$ and $V^*$ to denote the two parts of
the decomposition $V=\langle V \rangle +V^*$, where the part
$\langle V \rangle$ satisfies the transport equation
$(\d_t+\r\d_y)\langle V \rangle=0$ with the same group velocity $\r$
as the leading term $V_{01}$. To solve (\ref{dsch2e1}),  it is
sufficient to solve the following two systems:
\begin{equation}\label{dsch2e2}
\left\{ \begin{aligned}
         &-\r\d_y\langle E_{10}\rangle-(e_1 \times)J_1\d_y\langle H_{10} \rangle -(e_1 \times)J_2\d_y\langle M_{10} \rangle=0,\\
                  &-\r\d_y\langle H_{10}\rangle+J_1(e_1 \times)\d_y\langle E_{10} \rangle =A_0,\\
                 &-\r\d_y\langle M_{10}\rangle+J_2(e_1 \times)\d_y \langle E_{10} \rangle=-A_0
                          \end{aligned} \right.
                          \end{equation}
and
\begin{equation}\label{ds ch2e3}
\left\{ \begin{aligned}
         &\d_t E_{10}^*-(e_1 \times)J_1 \d_y H_{10}^*  -(e_1 \times)J_2 \d_yM_{10}^* =0,\\
                  &\d_t H_{10}^*+J_1(e_1 \times)\d_yE_{10}^*  =0,\\
                 &\d_t M_{10}^*+J_2(e_1 \times)\d_yE_{10}^*  =0.
                          \end{aligned} \right.
                          \end{equation}

By taking $-\r\d_y$ to the second and third equations of
(\ref{dsch2e2}), we obtain that
\begin{equation}\label{ds ch2e4}
\left\{ \begin{aligned}
                  &\r^2\d_y^2 \langle H_{10}\rangle+J_1(e_1 \times)
                  [(e_1 \times)J_1\d_y^2\langle H_{10}\rangle
                  +(e_1 \times)J_2 \d_y^2\langle M_{10}\rangle] =-\r\d_yA_0,\\
                 &\r^2\d_y^2 \langle M_{10}\rangle+J_2(e_1 \times)
                 [(e_1 \times)J_1\d_y^2\langle H_{10}\rangle  +(e_1 \times)J_2\d_y^2\langle M_{10}
                 \rangle]=\r\d_yA_0.
                          \end{aligned} \right.
                          \end{equation}

~

Since $V_{10}=(E_{10},H_{10},M_{10})\in \ker L_0$, by (\ref{kerl0
0}), we may suppose that
\begin{equation}\label{pt v100}
H_{10}=\begin{pmatrix}h_1\\h_2\\h_3\end{pmatrix},~
M_{10}=\begin{pmatrix}m_1\\h_2\\h_3\end{pmatrix}.
\end{equation}
Then by (\ref{def e1}) and (\ref{j1 j2}), the equation (\ref{ds
ch2e4}) becomes
\begin{equation}\label{ds ch2e}
\left\{ \begin{aligned}
                  &\r^2\d_y^2\langle h_1\rangle =-\frac{4k\de\r}{\o}(1-\frac{k\r}{\o})\d_y^2|g|^2,\\
                 &(\r^2-\frac{1}{2})\d_y^2\langle h_2\rangle=(\r^2-\frac{1}{2})\d_y^2\langle h_3\rangle=0,\\
                 &\r^2\d_y^2\langle m_1\rangle=\frac{4k\de\r}{\o}(1-\frac{k\r}{\o})\d_y^2|g|^2.\\
                          \end{aligned} \right.
                          \end{equation}

A solution to (\ref{ds ch2e}) is
\begin{equation}\label{sl ch2e}
\langle h_1\rangle =-\frac{4k\de\r}{\o
}(1-\frac{k\r}{\o})|g|^2,\quad \langle m_1\rangle=\frac{4k\de}{\o
\r}(1-\frac{k\r}{\o})|g|^2,\quad
 \langle h_2\rangle=\langle h_3\rangle=0.
                          \end{equation}

Plugging (\ref{pt v100}) and (\ref{sl ch2e}) into the first equation
of (\ref{dsch2e2}), we have that $\d_y \langle E_{10}\rangle=0$
which gives the trivial solution $\langle E_{10}\rangle=0$.

\medskip
We solve the equation (\ref{ds ch2e3}) by choosing the trivial
solution $E_{10}^*=H_{10}^*=M_{10}^*=0$. By doing so, we will not
see a source term in \eqref{f trans} and there will be no secular
growth for $\pi_1V_{11}$.

\medskip

With our choice for $V_{10}^*$ and $\langle V_{10}\rangle$, we have
a solution to (\ref{ds 2eo}):
\begin{equation}\label{sl v10}
V_{10}=\begin{pmatrix}E_{10}\\H_{10}\\M_{10}\end{pmatrix}=\frac{
4k\de }{\o
\r}(1-\frac{k\r}{\o})\begin{pmatrix}0\\-e_1\\e_1\end{pmatrix}|g|^2,
\end{equation}
where the vector $e_1$ is given in (\ref{def e1}).

Back to (\ref{eq u201}), by (\ref{exp v01}), (\ref{exp u11}) and
(\ref{sl v10}), we obtain that $ L_0V_{20}=0,$ which admits the
trivial solution $V_{20}=0$. Now equation (\ref{eq u201}) is solved.

\medskip

\fbox{\textbf{p=1}}~~For p=1, the equation (\ref{eq u2p}) becomes
\begin{equation}\label{eq u21}
\d_\t V_{01}+(\d_t+A(e_1))\d_y V_{11}+L_1V_{21}=2B(V_{01},V_{10}).
\end{equation}

Applying $\pi_1$ to \eqref{eq u21} gives
\begin{equation}\label{ds 1e}
\d_\t V_{01}+\pi_1(\d_t+A(e_1))\d_y V_{11}=2\pi_1B(V_{01},V_{10}).
\end{equation}
By \eqref{1-pi u11}, we have that

$$\pi_1(\d_t+A(e_1))\d_y V_{11}=(\d_t+\pi_1A(e_1)\d_y\pi_1)\pi_1 V_{11}-\pi_1A(e_1)\d_y L_1^{-1}\big(A(e_1)\d_yV_{01}\big),$$

We already have $\pi_1A(e_1)\d_y\pi_1=\r \d_y$ with group velocity
$\r$ in \eqref{gv}. In diffractive optics approximation, we have
$\displaystyle {-\pi_1A(e_1)\d_y
L_1^{-1}A(e_1)\d_y=-\frac{i}{2}\o''(k)\d_y^2},$ which gives a
Schr\"odinger equation. (See Proposition 2.6 in \cite{T0} or
Proposition 4.1 in \cite{JMR_Ind}). Here we compute the
Schr\"odinger operator $-\pi_1A(e_1)\d_y L_1^{-1}A(e_1)\d_y$
explicitly.

\medskip

By (\ref{exp pi1}), the equation (\ref{ds 1e}) is equivalent to
\begin{equation}\label{ds 1ein}
(\d_\t
V_{01}|W_0)+((\d_t+A(e_1)\d_y)V_{11}|W_0)=2(B(V_{01},V_{10})|W_0).
\end{equation}
\medskip
 Since $V_{01}=gW_0$, we have $(\d_\t V_{01}|W_0)=\d_\t g
|W_0|^2$. By (\ref{kerL1}), (\ref{eq trans}) and (\ref{exp u11}), we
have
$$(\d_t V_{11}|W_0)=|W_0|^2\d_tf+\frac{2ik\r}{\o^2}(1-\frac{k\r}{\o})(1-2\g)\d_y^2g$$
and
$$(A(e_1)\d_y V_{11}|W_0)=\frac{4k}{\o}\d_y f-\frac{2i}{\o}(1-\frac{k\r}{\o})\d_y^2g.$$
 For the right hand side of (\ref{ds 1ein}), by (\ref{def
B}), (\ref{exp pi1}) and (\ref{sl v10}),
$$2(B(V_{01},V_{10})|W_0)=\frac{ 8ik }{\o
\r}(1-\frac{k\r}{\o})(1-\g^2)g|g|^2.$$

We denote the real constant
\begin{equation}\label{ct D}
\nu:=\frac{|W_0|^2}{2}=1+\frac{k^2}{\o^2}+(1-\frac{k^2}{\o^2})^2,\nonumber
 \end{equation}
then the equation (\ref{ds 1e}) is equivalent to
\begin{equation}\label{ds 1ech}
\d_\t g+\frac{i}{\nu \o
}[\frac{k\r}{\o}(1-2\g)-1](1-\frac{k\r}{\o})\d_y^2g+(\d_t+\r \d_y)f
=\frac{ 4ik }{\o \r}(1-\frac{k\r}{\o})(1-\g^2)g|g|^2.
\end{equation}

We decompose (\ref{ds 1ech}) into two equations: a transport
equation in $f$ related to $\pi_1 V_{11}$:
\begin{equation}\label{f trans}
 (\d_t+\r \d_y)f=0,
 \end{equation}
and a cubic Schr\"odinger equation in $g$ related to $V_{01}$:
\begin{equation}\label{schr}
\d_\t g+i\nu_1\d_y^2g =i\nu_2g|g|^2,
 \end{equation}
where the real constants $\nu_1$ and $\nu_2$ are defined as
\begin{equation}\label{coe di}
\nu_1=\frac{1}{\nu \o
}\big[\frac{k\r}{\o}(1-2\g)-1\big]\big(1-\frac{k\r}{\o}\big),\quad
\nu_2=\frac{ 4k }{\o \r}\big(1-\frac{k\r}{\o}\big)(1-\g^2).\nonumber
\end{equation}

 Since the scalar function $g$ satisfies the
transport equation (\ref{eq trans}), then $g$ is of the form
\begin{equation}\label{g_1}
g(\t, t,y)=g_1(\t,z)|_{z=y-\r t},
\end{equation}
for some scalar function $g_1$. By \eqref{schr}, the function
$g_1(\t,z)$ solves
\begin{equation}\label{schr1}
\d_\t g_1+i\nu_1\d_z^2g_1 =i\nu_2g_1|g_1|^2.
 \end{equation}
By the classical theory of semilinear Schr\"odinger equations, given
any regular initial datum $g_1(\t,z)=a_0(z)\in H^s,~s>2+1/2$, the
equation \eqref{schr1} admits a unique local-in-time solution, with
existence time $T_1^*>0$ independent of $\e$, depending on the $H^s$
norm of $a_0$. One has the following estimate for all $T<T_1^*$:
$$
\sup_{0\leq \t \leq T} \|g_1(\t,\cdot)\|_{H^s} \leq C(T)<+\infty.
$$
Then there exists a unique solution $g\in
L^\infty([0,T_s^*[_\t\times \R_t,H^s)$ to \eqref{eq trans} and
\eqref{schr} which takes the form \eqref{g_1}.

\smallskip

Back to equation (\ref{eq u21}), one can obtain the solution
$V_{21}=(E_{21},H_{21},M_{21})$ in terms of the functions $f,~g$ and
an unknown scalar function $h$:
\begin{equation}\label{exp v21}
\left\{ \begin{aligned}
                 &E_{21}=-\frac{i\de k}{\o}h\O_0+\frac{\de}{\o}(-1+\frac{k\r}{\o})\d_yf \O_0+i\nu_3\d_y^2 g\O_0,\\
                 &H_{21}=h \O_0 ,\\
                 &M_{21}=-\g h
                 \O_0+\frac{2ik}{\o^2}(-1+\frac{k\r}{\o})\d_y f \O_0+\nu_4\d_y^2 g  \O_0.\\
                          \end{aligned} \right.
                          \end{equation}

The constants $m_3$ and $m_4$ are real and defined as
\begin{equation}
\nu_3=\frac{\de
\r}{2\o^2}(\frac{k\r}{\o}-1)\big(\frac{k\r}{\o}(1-2\l)-3\big),~~
\nu_4=-\frac{1}{\o^2}(\frac{k\r}{\o}-1)\big(\frac{4k\r}{\o}-
(\frac{k\r}{\o})^2 (1-2\l)-1\big).\nonumber
\end{equation}

We define $h$ and $f$ by
\begin{equation}\label{f h}
h=0,~f=i\big[\frac{2k}{\o^2}(-1+\frac{k\r}{\o})\big]^{-1}\nu_4 \d_y
g.
\end{equation}
Because the function $g$ satisfies the transport equation (\ref{eq
trans}), the function $f$ defined by \eqref{f h} is a solution to
\eqref{f trans}. The point is that, with this choice of $h$ and $f$,
there holds Lemma \ref{pt r} below. This is the key property that,
together with Lemma \ref{pt b}, allows the change of variable
\eqref{error rescale} in Section 4.2. By \eqref{exp v21}and \eqref{f
h}, we directly have
\begin{equation}\label{pt v21}
H_{21}=M_{21}=0.
\end{equation}

 \fbox{\textbf{p=2}}~~For $p=2$, the equation (\ref{eq u2p}) becomes $L_2V_{22}=2B(V_{01},V_{11}).$ By (\ref{exp v01}) and (\ref{exp
u11}), we have $B(V_{01},V_{11})=0$, and by the invertibility of
$L_2$, the unique solution is $V_{22}=0$.

\medskip

\fbox{\textbf{p$\geq$3}}~~For $p\geq3$, to equation (\ref{eq u2p}):
$V_{2p}=0$.

\end{subsection}

\begin{subsection}{The approximate solution and the remainder}

The vector space $\ker L(i(\o,k))$ is one-dimensional with generator
$W_0$ defined in (\ref{kerL1}). The initial leading amplitude $a(y)
\in \ker L(i(\o,k)) $ and $a\in H^s$ with $s>2+1/2$,  so there
exists a scalar function $a_0(y)\in H^s$ such that $a(y)=a_0(y)W_0$.

\smallskip

Given initial datum $g(0,0,y)=a_0(y)$, the transport equation
\eqref{eq trans} and the Schr\"odinger equation (\ref{schr}) admit a
unique solution which takes the form \eqref{g_1} over time interval
$[0,T_1^*)$ , where $T_1^*>0$ independent of $\e$ is the existence
time of the cubic Schr\"odinger equation \eqref{g_1}. For any
$T<T_1^*$ $g\in L^\infty([0,T]_\t \times \R_t, H^s)$. We then choose
$(f,h)$ as in (\ref{f h}), one has $f\in L^\infty([0,T]_\t \times
\R_t, H^{s-1})$ and the following estimates for any $T<T_1^*$:

\smallskip

By (\ref{exp v01}), we have $V_{01}\in L^\infty([0,T]_\t \times
\R_t, H^s)$. By (\ref{exp u11}), we have $V_{11}\in
L^\infty([0,T]_\t \times \R_t, H^{s-1})$. By (\ref{sl v10}), we have
$V_{10}\in L^\infty([0,T]_\t \times \R_t, H^{s})$. By (\ref{exp
v21}), we have $V_{21}\in L^\infty([0,T]_\t \times \R_t, H^{s-2})$.

\medskip

We then let
\begin{equation}\label{def va0}
\left\{ \begin{aligned}
                 &V^0(t,y,\th):=V_{01}(\e t,t,y)e^{i\th}+\overline{V}_{01}(\e t,t,y)e^{-i\th},\\
                 &V^1(t,y,\th):=\big(V_{10}
   +V_{11}e^{i\th}+\overline{V}_{11}e^{-i\th}\big)(\e t,t,y),\\
   &V^{2}(t,y,\th):=\big(V_{21}e^{i\th}+\overline{V}_{21}e^{-i\th}\big)(\e t,t,y).\\
                          \end{aligned} \right.
                          \end{equation}

Define $V^a(t,y,\th)$ as
\begin{equation}\label{def va}
V^a=V^0+\e V^1+\e^2 V^2,\nonumber
\end{equation}
then the profile $V^a$ satisfies
\begin{equation}\label{eq va}
\left\{ \begin{aligned}
                 &\d_tV^a+A(e_1)\d_yV^a+\frac{1}{\e}\{-\o\d_{\th}+A(e_1)k\d_{\th}+ L_0\}
V^a=B(V^a,V^a)+\e^2 R,\\
                 &V^a(0,y,\th)=V(0,y,\th)+\e b(y,\th)+\e^2 b_1(y,\th) \\
                          \end{aligned} \right.
                          \end{equation}
over long time interval $[0,T_1^*/\e[$. The initial perturbations
$b$ and $b_1$ have the expressions
\begin{eqnarray}\label{exp b}
&&b(y,\th)=-a_1(y,\th)+\big(V_{10}+V_{11}e^{i\th}+\overline{V}_{11}e^{-i\th}\big)(0,0,y),\\
\label{exp b1}&&b_1(y,\th)=-a_2(y,\th)+\big(V_{21}e^{i\th}+
\overline{V}_{21}e^{-i\th}\big)(0,0,y),
\end{eqnarray}
and the remainder $R(t,y,\th)$ is defined as
\begin{equation}\label{exp r}
R=-2B(V^0,V^2)-B(V^1,V^1)-\e2B(V^1,V^2)-\e^2B(V^2,V^2).
\end{equation}
For any $T<T^*_1$, we have the following estimates
\begin{equation}\label{est r b}
R\in L^\infty ([0,T/\e]_t,
H^1(\T_{\th},H^{s-2}(\R^1_y))),~(b,b_1)\in
H^1(\T_{\th},H^{s-2}(\R^1_y)).\nonumber
\end{equation}
At this stage, Proposition \ref{exis app1d} is proved.
\smallskip

We show more properties of $R$ and $b$ in the following two lemmas,
according to the WKB expansion in Section 2.1.
\begin{lem}\label{pt b}
 For initial perturbations $b$ and $b_1$ defined in \eqref{exp b}, we have
 the equivalence
\begin{equation}\label{ini diff2}
|\Pi_s(V(0)-V^a(0))|_{H^1(\T_{\th},H^{s-2}(\R^1_y))}=O(\e^2)
\Longleftrightarrow\Pi_s b=0 .
\end{equation}
Moreover, there holds $\Pi_s b=0$ if and only if the initial
corrector $a_1$ in {\rm \eqref{ini org}} is given by \eqref{exp a1}
below.

\end{lem}

\begin{proof}
By (\ref{pi0s}) and (\ref{exp u11}), we have that $\Pi_s V_{11}=0.$
Then $\Pi_s b=0$, by (\ref{exp b}), it is necessary and sufficient
to have
\begin{equation}\label{exp a1}
a_1(y,\th)=V_{10}(0,0,y)+\tilde{a}_1(y,\th)=\frac{ 4k\de }{\o
\r}(1-\frac{k\r}{\o})\begin{pmatrix}
0\\
- e_1\\
 e_1
\end{pmatrix}|a_0(y)|^2+\tilde{a}_1(y,\th),
\end{equation}
where $\tilde{a}_1$ satisfies $\Pi_s \tilde{a}_1=0$. The projector
$\Pi_s$ is defined in \eqref{pi0s}. Direct calculation gives the
equivalence \eqref{ini diff2}

\end{proof}

\begin{lem}\label{pt r}
With $f$ and $h$ given by {\rm(\ref{f h})}, the remainder $R$
defined in \eqref{exp r} satisfies
$$\Pi_s R=O(\e)~ {\rm in}~
L^\infty ([0,T/\e]_t, H^1(\T_{\th},H^{s-2}(\R^1_y))), \quad {\rm for
~any}~ T<T_1^*.$$
\end{lem}

\begin{proof}
 By (\ref{def va0}), we have
\begin{equation}
 B(V^0,V^2)=B(V_{11},\overline{V}_{21})+e^{2i\th}B
 (V_{11},V_{21})+c.c.\nonumber
\end{equation}
By \eqref{pt v21}, a consequence of \eqref{f h}, the components
$H_{21}$ and $M_{21}$ of $V_{21}$ satisfy $H_{21}=M_{21}=0$. Then by
the definition of $B$ in (\ref{def B}), it is easy to obtain
$B(V^0,V^2)=0.$ Also by (\ref{def va0}), we have that
\begin{equation}
 B(V^1,V^1)=B(V_{10},V_{10})+B(V_{11},\overline{V}_{11})+e^{i\th}B
 (V_{10},V_{11})+e^{i2\th}B
 (V_{11},V_{11})+c.c.\nonumber
\end{equation}

By (\ref{def B}), (\ref{exp v01}), (\ref{exp u11}) and (\ref{sl
v10}), direct calculation gives
\begin{equation}
B(V_{10},V_{10})=B(V_{11},V_{11})=0,~~
B(V_{11},\overline{V}_{11})=\begin{pmatrix}0\\A_1\\-A_1\end{pmatrix},\nonumber
\end{equation}
where
\begin{equation}
A_1=\frac{ik}{\o^2}(\frac{k\r}{\o}-1)(\bar{f}\d_y g +f \d_y \bar{g}
)(\O_0\times\overline{\O}_0).\nonumber
\end{equation}
With $f$ given in (\ref{f h}), it is easy to check that $\bar{f}\d_y
g +f \d_y \bar{g}=0$, and then $B(V_{11},\overline{V}_{11})=0$.
Through direct calculation, we have $\Pi_s B (V_{10},V_{11}) =0$.
Then $\Pi_s B(V^1,V^1)=0$, and by the definition of $R$ in
\eqref{exp r}, we have
\begin{equation}\label{new r}
\Pi_s R=-\e\Pi_s\big(2B(V^1,V^2)+\e B(V^2,V^2)\big).
\end{equation}
The lemma is proved.

\end{proof}

We constructed a WKB solution $v^a$ to \eqref{equ app}. We now
investigate the question whether the WKB solution actually
approximates the exact solution over an interval of existence.

\end{subsection}

\end{section}

\section{Maxwell-Bloch structure and long time existence.}
In this section, we first show that the system \eqref{1d sys pro}
has the Maxwell-Bloch structure, then prove a long time existence
result by normal form reduction.

\subsection {Spectral decomposition}
As mentioned in Section 1.2, for the symmetric matrix $A( e_1)\xi+
L_0/i$, we have the following spectral decomposition
\begin{equation}\label{sp dec}
A( e_1)\xi +\frac{L_0}{i}=\sum_{j=1}^9 \l_j(\xi)\Pi_j(\xi).
\end{equation}
The characteristic variety (that is, the union of the graphs $\xi
\mapsto \l_j(\xi)$) is pictured on figure 1. We have that for any
$\xi\in\R$, $ 1\leq j\leq6,$ and $7\leq j'\leq 9$:
\begin{equation}\label{pt eig v}
\l_1(\xi)\geq 2,\quad \l_2(\xi)\geq0,\quad 0\leq\l_3(\xi)<1, \quad
\l_j(\xi)=-\l_{7-j}(\xi),\quad \l_{j'}(\xi)=0.
\end{equation}
 For $1\leq j\leq 6$, we have the dispersion
relations
\begin{equation}\label{dis rel}
\xi^2=\frac{\l_j(\xi)+2\de_j}{\l_j(\xi)+\de_j}\l_j^2(\xi),\quad
\de_j=(-1)^j.
\end{equation}
\begin{figure}
\begin{center}
\psfragscanon
\psfrag{a}{$\l_1$}\psfrag{b}{$\l_2$}\psfrag{c}{$\l_3$}\psfrag{d}{$\l_4$}
\psfrag{e}{$\l_5$}\psfrag{f}{$\l_6$}\psfrag{w}{$\l$}\psfrag{k}{$\xi$}
\includegraphics[width=8cm]{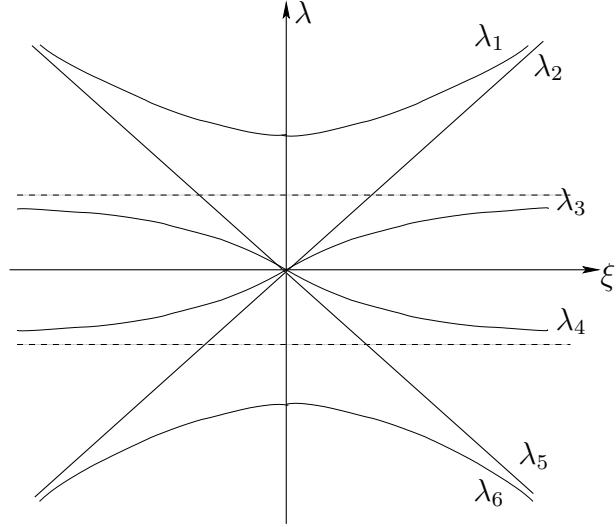}
\end{center}
\caption{The characteristic
 variety.}
\end{figure}

In order to reveal the Maxwell-Bloch structure, we define
\begin{equation}\label{v0 vs}
V_0=\Pi_0V=\sum_{j=1}^6\Pi_j(\e \D_y+k\D_\th) V ,\quad V_s=\Pi_s
V=\sum_{j=7}^9\Pi_j(\e \D_y+k\D_\th) V,
\end{equation}
where the total projectors $\Pi_0$ and $\Pi_s$ are defined in
\eqref{sp decom2} and the notation $\D:=\d/i.$ By (\ref{pi0s}), the
projector $\Pi_0$ and $\Pi_s$ are actually constant matrices. By
equation \eqref{sys prof} in $V$, the system in $(V_0,V_s)$ is
\begin{equation}\label{v0vs}
\left\{ \begin{aligned}
                 &\d_t V_0 +\frac{i}{\e}\mathcal{A}_0 V_0-\frac{\o \d_\th}{\e}V_0=\Pi_0B(V_0+V_s,V_0+V_s),\\
                 &\d_t V_s-\frac{\o \d_\th}{\e}V_s=\Pi_sB(V_0+V_s,V_0+V_s),\\
                          \end{aligned} \right.
                          \end{equation}
where the Fourier multiplier
\begin{equation}\label{A0}
\mathcal{A}_0=\sum_{j=1}^6
\l_j(\e\D_y+k\D_\th)\Pi_j(\e\D_y+k\D_\th).
                          \end{equation}

By \eqref{def B} and \eqref{pi0s}, we have
\begin{equation}
\Pi_0B(\Pi_0,\Pi_0)=\Pi_0B(\Pi_s,\Pi_s)=\Pi_sB(\Pi_s,\Pi_s)=\Pi_sB(\Pi_0,\Pi_s)=0.
\end{equation}
Then the equation \eqref{v0vs} becomes
\begin{equation}\label{v0vs1}
\left\{ \begin{aligned}
&\d_t V_0 +\frac{i}{\e}\mathcal{A}_0 V_0-\frac{\o \d_\th}{\e}V_0=2\Pi_0B(V_0,V_s),\\
&\d_t V_s-\frac{\o \d_\th}{\e}V_s=\Pi_sB(V_0,V_0).\\
                          \end{aligned} \right.
\end{equation}

\subsection{Strong transparency and normal form reduction}
The following proposition states that the system \eqref{v0vs1}
satisfies a strong transparency hypothesis.

\begin{prop}\label{trans}
There exists a constant $C>0$ such that for all $\xi,\eta \in \R$,
for any eigenvalues $\l_j(\xi),~\l_{j'}(\eta),1\leq j,j'\leq 6$, and
for all $u,~v\in \C^9$, one has
\begin{equation}\label{sttrans}
|\Pi_sB(\Pi_{j}(\xi)u,\Pi_{j'}(\eta)v)|\leq
C|\l_j(\xi)+\l_{j'}(\eta)||u||v|.
\end{equation}
\end{prop}

\begin{proof}

By direct calculation, for any $1\leq j\leq 6$, the eigenprojector
$\Pi_j$ has the formula
\begin{equation}\label{pij}
\Pi_j(\xi) =\frac{1}{|Q_j|^2}(~\cdot ~|Q_j(\xi))Q_j(\xi),
\end{equation}
where
\begin{equation}\label{Qj lj}
Q_j=\begin{pmatrix}\frac{-i\de_j\xi}{\l_j} \O_j\\\O_j\\-\g_j
\O_j\end{pmatrix},~~\O_j=\begin{pmatrix}0\\i\de_j\\1\end{pmatrix},~~\g_j=1-\frac{\xi^2}{\l_j^2},~\de_j=(-1)^j,
\end{equation}
and the couple $(\l_j,\xi)$ satisfies the dispersion
relation~(\ref{dis rel}). We then compute
\begin{equation}\label{psqj}
\Pi_sB(Q_j(\xi),Q_{j'}(\eta))=\frac{i}{2}
(\g_{j}(\xi)-\g_{j'}(\eta)(\de_{j}-\de_{j'})
\begin{pmatrix}0\\-e_1\\e_1\end{pmatrix}.
\end{equation}
By \eqref{Qj lj}, we have
\begin{eqnarray}\label{gj-gj'}
\g_{j}(\xi)-\g_{j'}(\eta)&=&(1-\frac{\xi^2}{\l_j(\xi)})-(1-\frac{\eta^2}{\l_{j'}(\eta)})\nonumber\\
&=&(1-\frac{\l_j(\xi)+2\de_j}{\l_j(\xi)+\de_j})-(1-\frac{\l_{j'}(\eta)+2\de_{j'}}{\l_{j'}(\eta)+\de_{j'}})\nonumber\\
&=&\frac{-\de_{j}(\l_{j'}(\eta)-\de_{j}\de_{j'}\l_j(\xi))}{(\l_j(\xi)+\de_j)(\l_{j'}(\eta)+\de_{j'})}.
\end{eqnarray}

Recall that $\de_{j}$ and $\de_{j'}$ can only take the value $1$ or
$-1$. If $\de_{j'}$ and $\de_j$ have the same sign,  by
\eqref{psqj}, we have $\Pi_sB(Q_j(\xi),Q_{j'}(\eta))=0$; if
$\de_{j'}$ and $\de_j$ have the opposite sign, by (\ref{gj-gj'}), we
have
\begin{equation}\label{psqj2}
\Pi_sB(Q_j(\xi),Q_{j'}(\eta))= \frac{i
(\l_{j'}+\l_j)}{(\l_j+\de_j)(\l_{j'}+\de_{j'})}
\begin{pmatrix}0\\-e_1\\e_1\end{pmatrix}.
\end{equation}
\ By \eqref{pij} and direct calculation, the equation
\eqref{sttrans} follows.
\end{proof}

With \eqref{v0vs1} and Proposition \ref{trans}, we see that system
\ref{v0vs1} has the Maxwell-Bloch structure (terminology introduced
in Section 1.1.2). This allows to use a nonlinear change of variable
introduced by Joly, M\'etivier and Rauch \cite{JMR_TMB}. Together
with the preparation condition $\Pi_s V(0)=O(\e)$, this gives
existence in time $O(1/\e)$ for the solution to the Cauchy problem
\eqref{v0vs1}. First, changing variable $V_s=\e W_s$ gives
\begin{equation}\label{v0vs2}
\left\{ \begin{aligned}
&\d_t V_0 +\frac{i}{\e}\mathcal{A}_0 V_0-\frac{\o \d_\th}{\e}V_0=2\e \Pi_0B(V_0,W_s),\\
&\d_t W_s-\frac{\o \d_\th}{\e}W_s=\frac{1}{\e}\Pi_sB(V_0,V_0).\\
                          \end{aligned} \right.
\end{equation}

Then we introduce the nonlinear change of variable
\begin{equation}\label{N}N=W_s- J(V_0,V_0),\end{equation}
where the symmetric bilinear form $J$ has the following form
\begin{equation}\label{J}
J(\sum_{p\in \Z} u_p e^{ip\th},\sum_{q\in \Z} v_q
e^{iq\th})=\sum_{p,q}J_{pq}(u_p,v_q)e^{i(p+q)\th},
\end{equation}
for some $J_{pq}$ to be determined below. The equation in $N$ is
\begin{eqnarray}
&&\d_t N-\frac{\o\d_\th}{\e}N=\d_t W_s-\frac{\o\d_\th}{\e}W_s-J(\d_t
V_0-\frac{\o\d_\th}{\e}V_0,V_0)-J\big(V_0,\d_t
V_0-\frac{\o\d_\th}{\e}V_0\big)\nonumber\\
&&~~~=\frac{1}{\e}\Pi_sB(V_0,V_0)-J\big(-\frac{i}{\e}\sum_{j=1}^6
\l_j(\e\D_y+k\D_\th)\Pi_j(\e\D_y+k\D_\th)
V_0,V_0\big)\nonumber\\
&&~~~~~~-J\big(V_0,-\frac{i}{\e}\sum_{j=1}^6
\l_j(\e\D_y+k\D_\th)\Pi_j(\e\D_y+k\D_\th) V_0\big)\ -2J\big(2\e
\Pi_0B(V_0,W_s),V_0\big)\nonumber.\nonumber
\end{eqnarray}

We choose the bilinear operator $J$ to eliminate the singular term
of order $O(1/\e)$ in the above equation. We consider the following
equation
\begin{eqnarray}
&&\Pi_sB(V_0,V_0)=-i J\big(\sum_{j=1}^6
\l_j(\e\D_y+k\D_\th)\Pi_j(\e\D_y+k\D_\th) V_0,V_0\big)\nonumber\\
&&~~~~~~~~~~~~~~~~~~~-i J\big(V_0,\sum_{j=1}^6
\l_j(\e\D_y+k\D_\th)\Pi_j(\e\D_y+k\D_\th) V_0\big)\nonumber.
\end{eqnarray}
Equivalently, in Fourier, for all ($\xi,\eta,p,q,j,j'$):
\begin{eqnarray}
&&\Pi_sB\big(\Pi_j(\e\eta+k p)V_{0p},\Pi_{j'}(\e(\xi-\eta)+k
q)V_{0q}\big)\nonumber\\
&&~~~~~~~~~~~~~~~~=-i
\l_j(\e\eta+k p)J_{pq}\big(\Pi_j(\e\eta+k p)V_{0p},\Pi_{j'}(\e(\xi-\eta)+k q)V_{0q}\big)\nonumber\\
&&~~~~~~~~~~~~~~~~~~~-i \l_j\big(\e(\xi-\eta)+k
q\big)J_{pq}\big(\Pi_j(\e\eta+k p)V_{0p},\Pi_{j'}\big(\e(\xi-\eta)+k
q\big)V_{0q}\big)\nonumber.
\end{eqnarray}
A solution to the above equation is given by
\begin{equation}\label{Jpq}
J_{pq}\big(\Pi_j(\xi) a,\Pi_{j'}(\eta)b\big):=i \sum_{j=1}^6
\sum_{j'=1}^6\frac{ \Pi_s B(\Pi_j( \xi) a, \Pi_{j'}(\eta)
b)}{\l_j(\xi)+\l_{j'}(\eta)},~\mbox{for all $a,b\in\C^9$}.
\end{equation}
 By Proposition \ref{trans}, $J$ is well defined and is
bounded from $H^{s_1}(\T_\th,H^{s_2}_y)\times
H^{s_1}(\T_\th,H^{s_2}_y)$ to $H^{s_1}(\T_\th,H^{s_2}_y)$ provided
$s_1>1/2,s_2>1/2$. The system in $(V_0,N)$ is now
\begin{equation}\label{mb3}
\left\{ \begin{aligned}
         &\d_t V_0 +\frac{i}{\e}\mathcal{A}_0 V_0-\frac{\o \d_\th}{\e}V_0=2\e \Pi_0B(V_0,N+J(V_0,V_0)), \\
         &\d_t N-\frac{\o\d_\th}{\e}N=-2\e J(\Pi_0B(V_0, N+J(V_0,V_0)), V_0).\\
                          \end{aligned} \right.
\end{equation}
We then rescale the time and define ${\displaystyle
(\mathcal{V}_0,\mathcal{N})(\t,y,\th)= (V_0,N)(\t/\e,y,\th)}$. The
system in $(\mathcal{V}_0,\mathcal{N})$ is
\begin{equation}\label{mb4}
\left\{ \begin{aligned}
         &\d_\t \mathcal{V}_0 +\frac{i}{\e^2}\mathcal{A}_0 \mathcal{V}_0-\frac{\o \d_\th}{\e^2}\mathcal{V}_0=
         2 \Pi_0B(\mathcal{V}_0,\mathcal{N}+J(\mathcal{V}_0,\mathcal{V}_0)), \\
         &\d_\t \mathcal{N}-\frac{\o\d_\th}{\e^2}\mathcal{N}=
         -2 J(\Pi_0B(\mathcal{V}_0, \mathcal{N}+J(\mathcal{V}_0,\mathcal{V}_0)), \mathcal{V}_0).\\
                          \end{aligned} \right.
\end{equation}
The initial datum is
\begin{equation}\label{ini ch1}\left\{
\begin{aligned} &\mathcal{V}_0(0,y,\th)=\Pi_0V(0,y,\th)=\Pi_0(e^{i\th}a(y)+e^{-i\th}\overline{a(y)}+\e
a_1(y,\th)+\e^2a_2(y,\th)),\\
&\mathcal{N}(0,y,\th)=W_s(0,y,\th)-
J(\mathcal{V}_0,\mathcal{V}_0)(0,y,\th)=\frac{1}{\e}\Pi_s
V(0,y,\th)-J(\mathcal{V}_0,\mathcal{V}_0)(0,y,\th).
  \end{aligned} \right.
\end{equation}

 Since the leading term $a(y)$ belongs to
$\ker L(i(\o,k))$, and $\o\neq 0$, we have $\Pi_s a=0$. Then the
initial datum for $\mathcal{N}$ is $O(1)$ in $H^1(\T_\th,
H^{s-2}_y)$:
$$
\mathcal{N}(0,y,\th)=\Pi_s
a_1(y,\th)+\e\Pi_s
a_2(y,\th)-J(\mathcal{V}_0,\mathcal{V}_0)(0,y,\th)\in H^1(\T_\th,
H^{s-2}_y).
$$

This gives well-posedness over diffractive times:
\begin{prop}\label{wp ch1}
The Cauchy problem {\rm(\ref{mb4})-(\ref{ini ch1})} admits a unique
solution on a time interval $[0,T_2^*[$, with $T_2^*>0$ independent
of $\e$. Moreover, for all $T<T_2^*$, we have the estimates :
\begin{equation}
\sup_{[0,T]}\|(\mathcal{V}_0,\mathcal{N})(\t)\|_{H^1(\T_\th,
H^{s-2}_y)}\leq C,\nonumber
\end{equation}
where the constant $C=C(s,T)$ is independent of $\e$ and depends on
$s$ through the sum of norms
$\|a\|_{H^s_y}+\|a_1\|_{H^1(\T_\th,H^{s-1}_y)}+\|a_2\|_{H^1(\T_\th,H^{s-2}_y)}.$
\end{prop}

\begin{proof}

The system \eqref{mb4} being symmetric hyperbolic, local-in-time
well-posedness is classical. Here the bounds are uniform in $\e$.
Indeed, the initial datum \eqref{ini ch1} is uniformly bounded with
respect to $\e$ in Sobolev spaces $H^1(\T_\th, H^{s-2}_y),~
s>2+1/2$. The $L^2$ estimate is uniform in $\e$ in spite of the
large $1/\e$ prefactor because the operator
$-\o\d_{\th}+A(e_1)k\d_{\th}+ L_0$ is skew-adjoint, and the
commutator estimates are trivial because the operator
$-\o\d_{\th}+A(e_1)k\d_{\th}+ L_0$ has constant coefficients.
\end{proof}

Back to the variable $(V_0,N)$, where $V_0$ and $N$ are introduced
in \eqref{v0 vs} and \eqref{N} respectively,  by the definition of
$(\mathcal{V}_0,\mathcal{N})$, we have that $(V_0,N)$ is well
defined over the long time interval $[0,T_2^*/\e[$. Back to the
variable $V$ introduced in \eqref{profile}, this gives
well-posedness over the time interval $[0,T_2^*/\e[$ with the
estimates for any $T< T_2^*$:
\begin{equation}\label{est V}
\sup_{[0,T/\e]}\|\Pi_0 V(t)\|_{H^1(\T_\th, H^{s-2}_y)}\leq
C,~~\sup_{[0,T/\e]}\|\Pi_s V(t)\|_{H^1(\T_\th, H^{s-2}_y)}\leq C \e,
\end{equation}
for the same constant $C=C(s,T)$ as in Proposition \ref{wp ch1}.

 We proved the long time existence of order $O(1/\e)$ of the
solution of the Cauchy problem \eqref{1d sys pro}. Back to the
original variable ${\displaystyle v(t,y)=V(t,y,\th)|_{\th=\frac{-\o
t+k y}{\e}}}$, the following corollary follows immediately.
\begin{corol}\label{cor exis}
Over the time interval $[0,T_2^*/\e[$, the Cauchy problem
{\rm\eqref{sys org}-\eqref{ini org}} admits a unique solution $v$ of
the form ${\displaystyle v(t,x)=V(t,x,\frac{ky-\o t}{\e})}$, with
$V(t,x,\th)\in L^\infty ([0,T/\e]_{t},
H^1(\T_{\th},H^{s-2}(\R^1_y)))$ for any $T<T_2^*$. Moreover, by
Sobolev embedding, we have $|v|_{L^\infty} \leq C$, where $C=C(s,T)$
is given in Proposition {\rm \ref{wp ch1}}.
\end{corol}

In the next section, we consider the stability of the WKB solution
from Section 2 and show two convergence results.

\section{Error estimates}

For the WKB solution $V^a$ from Section 2, we use the projection,
rescalling and normal form reduction from Section 3:
\begin{equation}
(V^a_0,V^a_s)=(\Pi_0V^a,\Pi_sV^a),\quad W^a_s=\frac{V^a_s}{\e},\quad
N^a=W^a_s-J(V^a_0,V^a_0),\quad (\mathcal{V}^a_0,\mathcal{N}^a)(\t)=
(V^a_0,N^a)(\t/\e).\nonumber
\end{equation}
where $J$ is defined by \eqref{J} and \eqref{Jpq}. Then the system
in $(\mathcal{V}^a_0,\mathcal{N}^a)$ is
\begin{equation}\label{eq app ch1}
\left\{ \begin{aligned}
         &\d_\t \mathcal{V}^a_0 +\frac{i}{\e^2}\mathcal{A}_0 \mathcal{V}^a_0-\frac{\o \d_\th}{\e^2}V^a_0
         =2 \Pi_0B(\mathcal{V}^a_0,\mathcal{N}^a+J(\mathcal{V}^a_0,\mathcal{V}^a_0))+\e\Pi_0R, \\
         &\d_\t \mathcal{N}^a-\frac{\o\d_\th}{\e^2}\mathcal{N}^a=-2 J(\Pi_0B(\mathcal{V}^a_0,
         \mathcal{N}^a+J(\mathcal{V}^a_0,\mathcal{V}^a_0)), \mathcal{V}^a_0)-2\e J(\mathcal{V}_0^a,\Pi_0 R)+ \Pi_s R\\
                          \end{aligned} \right.
\end{equation}
As shown in Section 2.2, an existence time for \eqref{eq app ch1} is
$T^*_1$ (introduced in Section 2 as an existence time for
\eqref{schr1}). Define the perturbations
\begin{equation}\label{def error}
\Phi(\t,y,\th)=(\mathcal{V}_0-\mathcal{V}_0^a)(\t
,y,\th),~~\Psi(\t,y,\th)=(\mathcal{N}-\mathcal{N}^a)(\t,y,\th),
\end{equation}
then the couple ($\Phi,\Psi$) solves the following system over time
interval $[0,T_{12}^*[~$ with $T_{12}^*=\min\{T_1^*,T_2^*\}$
($T_2^*$ is introduced in Proposition \ref{wp ch1} as an existence
time of $(\mathcal{V}_0,\mathcal{N}) $ for \eqref{mb4}-\eqref{ini
ch1}):
\begin{equation}\label{eq error}
\left\{
\begin{aligned}
         &\d_\t \Phi +\frac{i}{\e^2}\mathcal{A}_0\Phi-\frac{\o
         \d_\th}{\e^2}\Phi
         =2\Pi_0B(\mathcal{V}^a_0,\Psi)+H_0(\mathcal{V}^a_0,\mathcal{N}^a,\Phi,\Psi)-\e\Pi_0R, \\
         &\d_\t \Psi-\frac{\o\d_\th}{\e^2}\Psi=H_s (\mathcal{V}^a_0,\mathcal{N},\Phi,\Psi)+2\e J(\mathcal{V}_0^a,\Pi_0 R)-\Pi_s R,\\
                          \end{aligned} \right.
\end{equation}
where $H_0$ and $H_s$ are defined by
\begin{equation}\label{h0hs}
\left\{
\begin{aligned}
         &H_0(\mathcal{V}^a_0,\mathcal{N}^a,\Phi,\Psi):=2\Pi_0B(\Phi,\mathcal{N}^a+\Psi+J(\mathcal{V}_0^a+\Phi,\mathcal{V}_0^a+\Phi))+
          2\Pi_0B(\mathcal{V}^a_0,J(\Phi,\Phi+2\mathcal{V}^a_0)),\\
         & \begin{split}H_s(\mathcal{V}^a_0,\mathcal{N}^a,\Phi,\Psi):=
         &2 J(\Pi_0B(\mathcal{V}^a_0,\mathcal{N}^a+J(\mathcal{V}^a_0,\mathcal{V}^a_0)), \mathcal{V}^a_0) \\
          & \quad -2J(\Pi_0B(\mathcal{V}_0^a+\Phi,\mathcal{N}^a+\Psi
         +J(\mathcal{V}_0^a+\Phi,\mathcal{V}_0^a+\Phi)),\mathcal{V}^a_0+\Phi).\end{split}\\
                          \end{aligned} \right.
                          \end{equation}

By \eqref{eq va}, the initial datum for ($\Phi,\Psi$) is
\begin{equation}\label{ini error}
\left\{
\begin{aligned}
         &\Phi(0)=(V_0-V_0^a)(0)=-\e \Pi_0 b-\e^2 \Pi_0b_1, \\
         &\Psi(0)=\frac{1}{\e}(V_s-V_s^a)(0)-J(V_0,V_0)(0)+J(V^a,V^a)(0)=-\Pi_s b+\e b_2,\\
                          \end{aligned} \right.
\end{equation}
where $b$ and $b_1$ are defined by \eqref{exp b} and \eqref{exp b1},
and $b_2 \in H^1(\T_\th, H^{s-2}_y)$ is defined by
$$
b_2=-\Pi_s b_1+J(b+\e b_1,V_a(0))+J(V(0),b+\e b_1).
$$

We give two estimates for ($\Phi,\Psi$) in the following two
sections.

\subsection{First error estimate}
Here we assume only the polarization condition $a(y)\in \ker
L(i(\o,k))$. Then by \eqref{ini error}, we have
$\Phi(0)=O(\e),~\Psi(0)=O(1)$ in $ H^1(\T_\th, H^{s-2}_y)$. Then for
the symmetric hyperbolic system \eqref{eq error}, we have the
following proposition:

\begin{prop}\label{1st error est}
The Cauchy problem \eqref{eq error}-\eqref{ini error} admits a
unique solution on time interval $[0,T_3^*[$ with $T_3^*\geq
T_{12}^*$, and the following estimates hold for any $T<T_3^*$:
\begin{equation}\label{1st error}
         \|\Phi\|_{L^\infty([0,T],H^1(\T_\th,H^{s-2}_y))}\leq
         C(\e+T),~~
         \|\Psi\|_{L^\infty([0,T],H^1(\T_\th,H^{s-2}_y))}\leq
         C,\\
\end{equation}
where the constant $C=C(s,T)$ is independent of $\e$ and depends on
$s$ through the sum of norms
$\|a\|_{H^s_y}+\|a_1\|_{H^1(\T_\th,H^{s-1}_y)}+\|a_2\|_{H^1(\T_\th,H^{s-2}_y)}.$
\end{prop}

\begin{proof} As in the proof of Proposition \ref{wp ch1},
 the data being bounded in $\e$, we have existence, uniqueness,
 and uniform bounds in $\e$ for short times with existence time $T^*_3$ independent of $\e$.

 The source terms of \eqref{eq error} is
$O(1)$, then we have the estimate:
\begin{equation}\label{1st error1}
\|(\Phi,\Psi)\|_{L^\infty([0,T],H^1(\T_\th,H^{s-2}_y))}\leq C.
\end{equation}

 By the first equation of \eqref{ini error}, the initial datum $\Phi(0)=O(\e)$.
The classical $H^s$ estimate then gives
\begin{equation}\label{1st error2}
\|\Phi\|_{L^\infty([0,T],H^1(\T_\th,H^{s-2}_y))}\leq
C\big(|\Phi(0)|_{L^\infty([0,T],H^1(\T_\th,H^{s-2}_y))}+\int_0^T C
~dt \big)\leq C (\e+T)
\end{equation}

\medskip

Since we already know that ($\Phi,\Psi$) defined by \eqref{def
error} solves \eqref{eq error}-\eqref{ini error} over time
 interval $[0,T_{12}^*[$ where $T_{12}^*:=\min\{T_1^*,T_2^*\}$, we have $T_3^*\geq T_{12}^*$
 by uniqueness.
\end{proof}

\begin{remark}
Here we show first that the exact solution exists over long times
$O(1/\e)$ before showing that it is approximated by the  WKB
solution on the intersection of their intervals of existence.
Indeed, the Maxwell-Bloch structure allows us to find the normal
form of the nonlinear equations, in contrast to, e.g., the linear
normal form reduction of {\rm\cite{em3}}.
\end{remark}

\begin{remark}
For this first error estimate, we do not need the special choice of
$f$ given by \eqref{f h}. The estimate \eqref{1st error}, as well as
the first estimate \eqref{est1} in Theorem {\rm\ref{theorem}}, hold
for any regular function $f$ that satisfies the transport equation
\eqref{f trans}, e.g. $f=0$. (Recall, $f$ is introduced in Section
2.1 as a building block of the first corrector $V^1$ in WKB
approximation. See \eqref{exp u11}).
\end{remark}

Back to the original time and variables, we immediately obtain the
following corollary:
\begin{corol}\label{1st error cor}
For any $T<T_{12}^*$, we have the error estimate

\begin{equation}\label{ini error}
\left\{
\begin{aligned}
         \|\Pi_0(V-V^a)\|_{L^\infty([0,T/\e],H^1(\T_\th,H^{s-2}_y))}&\leq
C(\e+T),\\
         \|\Pi_s(V-V^a)\|_{L^\infty([0,T/\e],H^1(\T_\th,H^{s-2}_y))}&\leq
C \e.\\
                          \end{aligned} \right.
\end{equation}
The constant $C=C(s,T)$ is as in Proposition {\rm \ref{1st error
est}}.
\end{corol}

\subsection{Second error estimate}

In this section, we show the stability of the WKB solution in the
following sense: for some $T>0$ independent of $\e$, we have
$$\|\Pi_0(V-V^a)\|_{L^\infty([0,T/\e],H^1(\T_\th,H^{s-2}_y))}=O(\e),~~\|\Pi_s(V-V^a)\|_{L^\infty([0,T/\e],H^1(\T_\th,H^{s-2}_y))}=O(\e^2),$$
 provided the initial perturbation
\begin{equation}\label{ini
diff1}\|\Pi_0(V-V^a)(0)\|_{H^1(\T_\th,H^{s-2}_y)}=O(\e),~~\|\Pi_s(V-V^a)(0)\|_{H^1(\T_\th,H^{s-2}_y)}=O(\e^2).
\end{equation}
Here we assume that the initial corrector $a_1$ satisfies \eqref{exp
a1}. Then by Lemma \eqref{pt b},
$$
\|\Pi_s(V-V^a)\|_{L^\infty([0,T/\e],H^1(\T_\th,H^{s-2}_y))}=O(\e^2),
$$
and the initial datum of ($\Phi,\Psi$) (introduced in \eqref{def
error}) is
\begin{equation}\label{ini error2}
         \Phi(0)=-\e \Pi_0 b-\e^2 \Pi_0b_1=O(\e),\quad
         \Psi(0)=-\Pi_s b+\e b_2=\e b_2 =O(\e).
\end{equation}
 For the remainder $R$ (introduced in \eqref{eq va} and explicitly given in \eqref{exp r}), with the choice of $f$ in \eqref{f
h}, by Lemma \ref{pt r}, we have that $\Pi_s R=-\e R_s$ with
$R_s=2B(V^1,V^2)+\e B(V^2,V^2)\in
L^\infty([0,T/\e],H^1(\T_\th,H^{s-2}_y))$ for any $T<T^*_1$. We now
consider the rescaled variables
\begin{equation}\label{error rescale}
\Phi_1=\Phi/\e,~\Psi_1=\Psi/\e,
\end{equation}

then the equation in ($\Phi_1,\Psi_1$) is

\begin{equation}\label{eq error2}
\left\{
\begin{aligned}
         &\d_\t \Phi_1 +\frac{i}{\e^2}\mathcal{A}_0\Phi_1-\frac{\o
         \d_\th}{\e^2}\Phi_1
         =2\Pi_0B(\mathcal{V}^a_0,\Psi_1)+\frac{1}{\e}H_0(\mathcal{V}^a_0,\mathcal{N}^a,\e \Phi_1,\e\Psi_1)-\Pi_0R, \\
         &\d_\t \Psi_1-\frac{\o\d_\th}{\e^2}\Psi_1=\frac{1}{\e}H_s (\mathcal{V}^a_0,\mathcal{N},\e\Phi_1,\e\Psi_1)
         +2 J(\mathcal{V}_0^a,\Pi_0 R)-R_s,\\
                          \end{aligned} \right.
\end{equation}
with the initial datum
\begin{equation}\label{ini error3}
         \Phi_1(0)=- \Pi_0 b-\e \Pi_0b_1, \quad
         \Psi_1(0)=b_2,
\end{equation}
and where $(H_0,H_s)$ are defined in \eqref{h0hs}. Note that by
bilinearity of $B$, and pointwise bounds for the approximate
solution, we have uniform bounds $$ \frac{1}{\e} |H_0({\cal V}_0^a,
{\cal N}^a, \e \Phi, \e \Psi)| \leq C (|\Phi| + |\Psi|), \qquad
 \frac{1}{\e} |H_s({\cal V}_0^a,
{\cal N}^a, \e \Phi, \e \Psi)| \leq C (|\Phi| + |\Psi|).$$

It is now classical to deduce uniform bounds for $\Phi_1$ and
$\Psi_1$ in times $O(1)$:

\begin{prop}\label{2nd error est}
With the choice of $a_1$ in \eqref{exp a1}, the Cauchy problem
\eqref{eq error2}-\eqref{ini error3} admits a unique solution
{\rm($\Phi_1,\Psi_1$)} on $[0,T_4^*[$ with $T_4^*>0$ independent of
$\e$, and for any $T<T_4^*$:
\begin{equation}\label{2nd error}
         \|\Phi_1\|_{L^\infty([0,T],H^1(\T_\th,H^{s-2}_y))}\leq
         C,~~\|\Psi_1\|_{L^\infty([0,T],H^1(\T_\th,H^{s-2}_y))}\leq
         C,\\
\end{equation}
where the constant $C=C(s,T)$ is independent of $\e$ and depends on
$s$ through the sum of norms
$\|a\|_{H^s_y}+\|a_1\|_{H^1(\T_\th,H^{s-1}_y)}+\|a_2\|_{H^1(\T_\th,H^{s-2}_y)}.$

\end{prop}

\begin{remark}
For this second error estimate, we do need the special choice of $f$
in \eqref{f h}. Together with the choice of initial corrector $a_1$
in \eqref{exp a1}, it allows us to rescale the solution \eqref{error
rescale} to obtain the desired estimates.

\end{remark}

Back to the original time and variables, we immediately deduce the
following corollary:
\begin{corol}\label{2nd error cor}
Let $T^*=\min\{T_{12}^*,T_4^*\}$, then for any $T<T^*$, we have the
error estimates
\begin{equation}\label{1st error1}
         \|\Pi_0(V-V^a)\|_{L^\infty([0,T/\e],H^1(\T_\th,H^{s-2}_y))}\leq
         C\e,~~\|\Pi_s(V-V^a)\|_{L^\infty([0,T/\e],H^1(\T_\th,H^{s-2}_y))}\leq
         C \e^2,
\end{equation}
where the constant $C=C(s,T)$ is as in Proposition {\rm \ref{2nd
error est}}.
\end{corol}

This shows stability, that is, the WKB approximate profile stays
close to the exact profile over its existence time with an error
estimate that is comparable to the initial error.

\subsection{Proof of Theorem \ref{theorem}}
We now sum up and prove Theorem \ref{theorem}.

\medskip

First, by Corollary \ref{cor exis}, the function
$v(t,y)=V(t,y,\th)|_{\th=\frac{ky-\o t}{\e}}$ solves \eqref{sys
org}-\eqref{ini org} over time interval $[0,T^*[$, and this gives
the first result of Theorem \ref{theorem}.

\medskip

Second, by the Corollary \ref{1st error cor} and Sovolev embedding
$H^1(\T_\th,H^{s-2}_y)\subset L^\infty(\T_\th\times\R_y)$, the error
estimates \eqref{est1} and \eqref{est11} of Theorem \ref{theorem}
follow immediately.

\medskip

Finally, by the Corollary \ref{2nd error cor} and Sovolev embedding
$H^1(\T_\th,H^{s-2}_y)\subset L^\infty(\T_\th\times\R_y)$, we have
the second error estimate \eqref{est2}. This shows stability of the
approximate solution.

~~

\textbf{Acknowledgment}. I warmly thank Eric Dumas and Isabelle
Gallagher for stimulating discussions. I particularly thank my
advisor Benjamin Texier for his fruitful remarks and for suggesting
this interesting subject.

~

\end{document}